\documentclass[journal]{IEEEtran}
\usepackage{amsfonts,amsmath,amssymb,graphicx,color}
\usepackage[dvipsnames]{xcolor}
\usepackage{amsthm}
\usepackage{xfrac}

\newcommand{\ew}[1]{{\color{black}#1}}
\newcommand{\ndgd}{\textcolor{black}{NEAR-DGD }}
\newcommand{\ndgdt}{\textcolor{black}{NEAR-DGD$^t$ }}
\newcommand{\ndgdp}{\textcolor{black}{NEAR-DGD$^+$ }}
\newcommand{\norm}[1]{\left|\left|#1\right|\right|}

\newtheorem{thm}{Theorem}[section]

\newtheorem{lem}[thm]{Lemma}
\newtheorem{cor}[thm]{Corollary}
\newtheorem{assum}[thm]{Assumption}
\numberwithin{equation}{section}

\ifCLASSINFOpdf
\else
\fi
\begin{document}
%
\title{Balancing Communication and Computation in Distributed Optimization}
%
%
%

\author{Albert~S.~Berahas,
        Raghu~Bollapragada,
        Nitish~Shirish~Keskar,
        and~Ermin~Wei,~\IEEEmembership{Member,~IEEE}
\thanks{A. S. Berahas was with the Department of Industrial Engineering and Management Sciences, Northwestern University,
       Evanston, IL, USA. (email: {albertberahas@u.northwestern.edu})}
\thanks{R. Bollapragada was with the Department of Industrial Engineering and Management Sciences, Northwestern University,
       Evanston, IL, USA. (email: {raghu.bollapragada@u.northwestern.edu})}
       \thanks{N. S. Keskar was with Salesforce Research, Palo Alto, CA, USA. Work was performed when author was at Northwestern University. (email: {keskar.nitish@u.northwestern.edu})}
\thanks{E. Wei was with the Department of Electrical Engineering and Computer Science, Northwestern University,
       Evanston, IL, USA. (email: {ermin.wei@northwestern.edu})}}

\maketitle

\begin{abstract}
Methods for distributed optimization have received significant attention in  recent years owing to their wide applicability in various domains including machine learning, robotics and sensor networks. A distributed optimization method typically consists of two key components: communication and computation. More specifically, at every iteration (or every several iterations) of a distributed algorithm, each node in the network requires some form of information exchange with its neighboring nodes (communication) and the computation step related to a (sub)-gradient (computation). The standard way of judging an algorithm via only the number of iterations overlooks the complexity associated with each iteration. Moreover, various applications deploying distributed methods may prefer a different composition of communication and computation.

Motivated by this discrepancy, in this work we propose an adaptive cost framework which adjusts the cost measure depending on the features of various applications. We present a flexible algorithmic framework, where communication and computation steps are explicitly decomposed to enable algorithm customization for various applications. We apply this framework to the well-known distributed gradient descent (DGD) method, and show that the resulting customized algorithms, which we call DGD$^t$, NEAR-DGD$^t$ and NEAR-DGD$^+$, compare favorably to their base algorithms, both theoretically and empirically. The proposed NEAR-DGD$^+$ algorithm is an exact first-order method where the communication and computation steps are nested, and when the number of communication steps is adaptively increased, the method converges to the optimal solution. We test the performance and illustrate the flexibility of the methods, as well as practical variants, on quadratic functions and classification problems that arise in machine learning, in terms of iterations, gradient evaluations, communications and the proposed cost framework.
\end{abstract}

\begin{IEEEkeywords}
Distributed Optimization, Communication, Optimization Algorithms, Network Optimization
\end{IEEEkeywords}

%
\IEEEpeerreviewmaketitle

\section{Introduction}
\label{sec:intro}
\IEEEPARstart{T}{he} problem of optimizing an objective function by employing a distributed procedure using multiple agents in a connected network has gained significant attention over the last years. This is motivated by its wide applicability to many important engineering and scientific domains such as, wireless sensor networks \cite{ling2010decentralized,predd2006distributed,giannakisAdHoc,zhao2002information}, multi-vehicle and multi-robot networks \cite{cao2013overview,ren2007information,zhou2011multirobot}, smart grids \cite{giannakis2013monitoring,kekatos2013distributed} and machine learning \cite{duchi2012,tsianos2012consensus}. In such problems, each agent (or node) has access to a component of the overall objective function and can only communicate with its neighbors in the underlying network. The collective goal is to minimize the summation of individual components. Formally, the system-wide problem can be represented as
\begin{align}		\label{eq:prob}
	\min_{x\in \mathbb{R}^p}\quad h(x) =  \sum_{i=1}^n f_i(x),
\end{align}
where $n$ represents the number of agents in the network, convex function $h: \mathbb{R}^p \rightarrow \mathbb{R}$ is the {\it global objective function}, convex function $f_i: \mathbb{R}^p \rightarrow \mathbb{R}$ for each \textcolor{black}{$i\in \{1,2,...,n \}$} is the {\it local objective function} available only to node $i$, and vector $x\in \mathbb{R}^p$ is the decision variable that the agents are optimizing cooperatively. This setup naturally calls for {\it distributed (optimization) algorithms}, where the agents iteratively perform local {\it computations} based on a local objective function and local {\it communications}, i.e., information exchange with their immediate neighbors in the underlying network, to solve the system-wide problem \eqref{eq:prob}.

In order to decouple the computation of individual agents, 
problem \eqref{eq:prob} is often reformulated as the  following consensus optimization problem by introducing a local copy $x_i \in \mathbb{R}^p$ for each agent $i\in \{1,2,...,n \}$ \cite{bertsekas1989parallel,nedic2009distributed},
\begin{align}		\label{eq:cons_prob}
	\min_{x_i \in \mathbb{R}^p}&\quad \sum_{i=1}^n f_i(x_i)\\
	 \text{s.t.} &\quad  x_i = x_j, \quad \forall i, j \in \mathcal{N}_i, \nonumber
\end{align}
where $\mathcal{N}_i$ denotes the set of neighbors of agent $i$. In this formulation, the local objective function of the $i^{th}$ agent only depends on the local copy $x_i$. An equality constraint, often referred to as the {\it consensus} constraint, is imposed to enforce that the local copies of neighboring nodes are equal. Since the underlying network is connected and the consensus constraint ensures that all local copies are equal,  problems \eqref{eq:prob} and \eqref{eq:cons_prob} are equivalent.

While there is a proliferating literature on developing distributed optimization methods for the above problem, most
follow the conventional approach of tracking the number of iterations to judge the efficiency of a distributed algorithm, i.e., the best algorithm achieves optimality in the minimal number of iterations,
and overlook the complexity associated with each iteration. In this work, we propose an alternative metric, an adaptive cost framework (in Section \ref{sec:cost}) to account for the different environments and hardware constraints of various applications, where distributed optimization methods are used. In this new cost framework, we consider communication and computation costs separately and weigh them using parameters specific to the environment. This 
cost framework also motivates our development of a class of flexible distributed methods, where we decompose communication and computation steps. This new class of algorithms is then customizable depending on the application.

\subsection{Literature Review}\label{sec:lit}

Our work is related to the growing literature on distributed algorithms for solving problem (\ref{eq:cons_prob}). We outline the various lines of research next.  One class of methods build upon the seminal works \cite{bertsekas1989parallel,TsitsiklisThes}, which proposed a parallel computation framework. In \cite{nedic2009distributed}, the authors introduced a first-order primal iterative method, known as {\it distributed (sub)-gradient descent (DGD)}. In one step of DGD, each agent updates its estimate of the solution via a linear combination of a gradient descent step with respect to its local objective function and a weighted average with local neighbors (also known as a consensus step). A number of later contributions \cite{bo6,aliLinMorse, MouraFastGrad, JJSubgradient,LORandomNetwork,LobelOzdaglarFeijer,MateiBaras,ne11,NOOTQuantization,AsuChapter,averagepaper,NOPConstrained,SNVStochasticGradient,ra10, extra,srivastava2011distributed} extended DGD to other settings, including stochastic networks, constrained problems, and noisy environments. {\it Coordinate descent} type methods, in either primal or dual space, have also been used in the distributed setting \cite{hong2016unified,pe16,richtarik2016parallel}. Another line of research is based on Nesterov's {\it dual averaging} algorithm \cite{NesterovDualAvg}, whose distributed version was proposed and analyzed in \cite{duchi2012}. 
{\it Dual decomposition based} methods have also recently  gained much attention. This class of methods includes augmented Lagrangian methods and Alternating Direction Method of Multiplier (ADMM) \cite{Bertsekas2012Survey,ADMMBoyd,chang2015asynchronous,Eckstein2012,PontusDiagADMM, IutzelerRand, MotaColoring,Mota2012,giannakisAdHoc,WeiCDCADMM, wei2013AsynADMM,proxSVRG,GiannakisChannelDecoding}. The last category includes {\it second-order methods}, where Newton-type methods are used to obtain faster rates of convergence \cite{eisen2017decentralized,Mansoori17Superlinear,mokhtariNetworkNewtonI, mokhtariNetworkNewtonII}. All these methods adopt the standard iteration count metric.


Closely related to our work are a few very recent contributions that incorporate communication considerations in the design of distributed algorithms \cite{chow2016expander,lan2017communication,shamir2014communication, tsianos2012communication,zhang2015disco, zhang2012communication}. \ew{In \cite{lan2017communication,shamir2014communication,zhang2015disco, zhang2012communication}, the goal was to develop one particular method, where the number of communication steps is reasonable compared to the iteration complexity.} While these methods are communication-efficient (with respect to some metric), they lack the flexibility to adapt to different environments. In \cite{chow2016expander}, the authors consider how to design a network topology that is communication-efficient, assuming the network topology can be controlled.  The closest work to ours is \cite{tsianos2012communication}. In this recent work, the authors control the frequency of communication steps and analyze the time performance of a dual averaging based algorithm. The goal of \cite{tsianos2012communication} was to show that speed-up is possible when the agents communicate less frequently. While this work pioneers the idea of adjusting the relative frequencies of communication and computation, it focuses solely on reducing the runtime of the algorithm and overlooks other important aspects, such as energy consumption.


In this paper, in order to demonstrate our new cost framework \eqref{eq:cost}, we consider decoupling the communication and computation steps of the well-studied DGD method and vary the frequency of these steps. On this front, our work is related to  \cite{sayed2013diffusion, chen2012fast, MouraFastGrad}. {In \cite{chen2012fast}, the authors propose to increase the number of communication steps at rate $k$, where $k$ is the number of iterations, to ensure convergence with a fixed stepsize (aka steplength) for proximal gradient-based algorithms on composite nonsmooth convex problems.} In \cite{MouraFastGrad}, the authors extended this idea to smooth problems and propose to increase communication at rate $\log(k)$.

In \cite{sayed2013diffusion}, the author propose two algorithms for quadratic problems: CTA (Combine-then-Adapt) and ATC (Adapt-then-Combine). Both of these methods stem from a new way of combining communication and computation steps and they differ in the order that the consensus and gradient operations are performed. {While our method employs a similar way to decompose and combine communication and computation steps, our method is not restricted to using exactly one consensus step and one gradient step at each iteration. Hence, it is more general and offers flexibility in how to combine these steps.} {In particular, our convergence guarantee holds for many different ways of increasing the frequency of communication including $k$ and $\log(k)$, as appeared in \cite{chen2012fast} and \cite{MouraFastGrad}, respectively.}


One major problem of the standard DGD method is that it only converges to a neighborhood of the optimal solution when a constant stepsize is employed. Recently, there have been many new distributed algorithms \cite{extra, di2016next,nedich2016achieving, qu2017harnessing,li2017decentralized, eisen2017decentralized} that can achieve exact convergence with a constant stepsize. In \cite{extra, qu2017harnessing,nedich2016achieving, eisen2017decentralized}, the authors also show their respective algorithms can achieve a linear rate of convergence. Our work is also related to this line of literature as many instances of our proposed class of algorithms achieve exact convergence. Moreover, one instance of our method converges linearly with respect to number of gradient computations, and sublinearly with respect to number of communications or our cost framework. While the proposed method has many similarities with the existing algorithms, it also boasts unique flexibility and adaptability.

\subsection{Contributions}

Our innovations in this paper are on two fronts:
(i) the new adaptive cost framework, and
(ii) the proposed class of flexible algorithms motivated by this framework. We next describe our main contributions:
\begin{itemize}
	\item We introduce a metric of performance based on the weighted combination of costs of both communication and computation steps. This metric is adaptive to, and can accurately characterize, different features of the application environments independent of the distributed algorithms.
\item
We decompose the communication and computation steps of DGD to enable algorithm customization. Based on this, we propose three classes of related flexible algorithms, which we call DGD$^t$, \ndgdt and NEAR-DGD$^+$. We can tune the  instances in these classes to balance the communication and computation costs according to the application.
\item We develop a class of exact first-order methods with constant stepsize (
NEAR-DGD$^+$), based on nesting the communication and computation steps, and increasing the number of consensus steps performed as the algorithm evolves. When we increase the number of consensus steps at rate $k$, where $k$ is the number of iterations, then we obtain an algorithm that achieves exact convergence at a linear rate. In particular, to get an $\epsilon$-accurate solution, we need $\mathcal{O}(\log(\sfrac{1}{\epsilon}))$ numbers of gradient computations and $\mathcal{O}(\log(\sfrac{1}{\epsilon})^2)$ number of communication rounds.
\item We illustrate the empirical performance of some instances of the proposed class of methods on quadratic and logistic regression problems as measured by
our new cost framework. We also demonstrate some practical instances of the class of methods that perform very well in practice in terms of iterations, number of communications, number of gradients and combined cost.
\end{itemize}

{In summary, our main contribution is the proposed cost framework in conjunction with the decomposition of the communication and computation steps. This allows for flexibility in algorithmic design, for a class of theoretically sound and efficient algorithms, and the first step towards harmonization of the communication and computation costs.}

The paper is organized as follows. In Section \ref{sec:cost} we describe in detail our proposed cost framework. Section \ref{sec:prelim} reviews relevant distributed optimization preliminaries such as reformulations of \eqref{eq:cons_prob}, and the DGD method. The variant of DGD method with multiple consensus steps, which we call DGD$^{t}$, is introduced and analyzed in Section \ref{sec:DGDt}. In Section \ref{sec:dgd_variants}, we describe the new \ndgdt and \ndgdp methods, provide theoretical analysis of the variants and also present numerical results. We provide some final remarks and future directions of research in Section \ref{sec:fin_remarks}.

\section{Adaptive Cost Framework}
\label{sec:cost}

A typical iteration of a distributed optimization method consists of some local computation (typically gradient or Hessian evaluation) and neighborhood communication.  While the amount of computation and communication per iteration differs from one algorithm to another, all iterations
are counted blindly as equal in the traditional iteration counting metric.  Moreover, as distributed algorithms are deployed in various contexts, the diverse range of scenarios calls for different ways to account for the cost (in terms of time, energy, or any other metric) of an algorithm. To illustrate this we discuss two motivating examples. Consider first the problem of controlling a swarm of battery powered robots, with low-energy computation modules onboard, connected via an energy-intense communication protocol. Since the robots have limited energy supply, communication steps can be very expensive, while longer task completion times may not be problematic. On the other hand, we consider solving a large-scale machine learning problem on a cluster of computers that are physically connected or with shared memory access. In this case, the cost of communication can be ignored (inexpensive in terms of time), while the computational cost (time) can be expensive depending on the size of the data set on each machine. Hence, a desirable metric should not only count the total number of communication and computation steps, it should also weigh the two appropriately according to different environments.

We propose an adaptive cost framework to evaluate the performance of distributed optimization methods which explicitly accounts for the cost of communication and computation, and can be customized depending on the specific application. In particular, we propose the following simple yet powerful metric
\begin{align}		\label{eq:cost}
\text{Cost} = \# \text{Communications} \times c_c + \# \text{Computations} \times c_g,
\end{align}
where $c_c$ and $c_g$ are exogenous application-dependent parameters reflecting the costs of communication and computation, respectively. For instance, when energy is the most constraining resource of the environment, the parameters $c_c$ and $c_g$ would reflect the energy consumed per step of communication/computation. Similarly, when the runtime is of concern,  the parameters would correspond to the time needed for a communication/computation step. This cost could also represent some combination of time and energy. \ew{In the battery powered robots example we would have $c_c>c_g$, and in the machine learning example we would have $c_c<c_g$.}  We note that if the cost of communication and computation of one iteration is $1$, then to design an algorithm with minimal cost reduces to the standard problem of finding algorithm with the least iteration count.

\section{Preliminaries}
\label{sec:prelim}

\ew{In this section, we introduce an equivalent compact reformulation of problem \eqref{eq:cons_prob} and review the basics of the DGD method, both of which will be used to build our class of flexible algorithms. }

\subsection{Equivalent Reformulations}

For compactness, we express problem \eqref{eq:cons_prob} as
\begin{align}		\label{eq:cons_prob1}
	\min_{x_i \in \mathbb{R}^p}&\quad f(\textbf{x}) = \sum_{i=1}^n f_i(x_i)\\
	\text{s.t.} & \quad (\textbf{W}\otimes I_p)\textbf{x} = \textbf{x} \nonumber
\end{align}
where $\textbf{x} \in \mathbb{R}^{np}$ is a concatenation of all local $x_i$'s and  $\textbf{W}$ is a matrix of size $\mathbb{R}^{n \times n}$, i.e.,
\begin{align*}
	\textbf{x} =
	\begin{bmatrix}
		x_{1}  \\
		x_{2}  \\
		\vdots  \\
		x_{n}
	\end{bmatrix},
	\quad
	\textbf{W} =
	\begin{bmatrix}
		w_{11}  & w_{12} & \cdots & w_{1n} \\
		w_{21} & w_{22} & \cdots & w_{2n} \\
		\vdots & \vdots & \ddots & \vdots\\
		w_{n1} & w_{n2} & \cdots & w_{nn}
	\end{bmatrix}.
\end{align*}
Matrix $I_p$ is the identity matrix of dimension $p$, and the operator $\otimes$ denotes the Kronecker product operation, with $\textbf{W}\otimes I_p \in \mathbb{R}^{np \times np}$. Moreover, matrix $\textbf{W}$ has the following properties: it is symmetric, doubly-stochastic, \ew{with diagonal elements $w_{ii}>0$ and  off-diagonal elements $w_{ij}>0$ ($i\neq j$) if and only if $i$ and $j$ are neighbors in the underlying communication network}. \ew{Matrix $\textbf{W}$ is known as the {\it consensus matrix} and it has the property that $(\textbf{W}\otimes I_p) \textbf{x}=\textbf{x}$ if and only if $x_i=x_j$ for all $i$ and $j$ in the connected network \cite{nedic2009distributed}, i.e., problems \eqref{eq:cons_prob} and \eqref{eq:cons_prob1} are equivalent.} We also have that matrix $\textbf{W}$ has exactly one eigenvalue equal to 1 and the rest of eigenvalues have absolute values strictly less than 1. We use $\beta$, with $0<\beta<1$, to denote the second largest magnitude of the eigenvalues of $\textbf{W}$. For the rest of the paper, we focus on developing methods to solve problem \eqref{eq:cons_prob1}.

\subsection{Distributed Gradient Descent}
\label{sec:dgd}

\ew{We now review the basic Distributed Gradient Descent (DGD) method \cite{nedic2009distributed}, which is the building block for our later development of the DGD$^t$, \ndgdt and \ndgdp methods.}
The DGD method is a first-order method for solving problem \eqref{eq:cons_prob1}, where each agent updates its local estimate iteratively using a gradient based on local information and information exchanged with its neighbors in the network.  The $k^{th}$ iteration of the DGD method for any node $i$ can be expressed as
\begin{align*}
	x_{i,k+1} = \sum_{j \in \mathcal{N}_i \cup \{i\}}w_{ij} x_{j,k} - \alpha \nabla f_i(x_{i,k}), \qquad \forall i=1,\ldots,n,
\end{align*}
\ew{where $x_{i,k}$ represents the local estimate of agent $i$ at iteration $k$ and the positive scalar $\alpha$ denotes the stepsize.}
Effectively, the $i^{th}$ agent computes a weighted average of its and its neighbors local estimates, and takes a step in the negative gradient direction obtained using only local information.
Equivalently, in the concatenated notation, the DGD method can be expressed as
\begin{align}		\label{eq:dgd_vec}
	\textbf{x}_{k+1} = \textbf{Z}\textbf{x}_k - \alpha \nabla \textbf{f}(\textbf{x}_k)
\end{align}
where
\begin{align*}
\nabla \textbf{f}(\textbf{x}_k) = \begin{bmatrix}
   		  \nabla f(x_{1,k})  \\
   		  \nabla f(x_{2,k})  \\
    		   \vdots  \\
   		  \nabla f(x_{n,k})
	\end{bmatrix} \in \mathbb{R}^{np}
\end{align*}
 and the matrix $\textbf{Z}= \textbf{W}\otimes I_p \in \mathbb{R}^{np \times np}$.


The DGD method can also be thought of as a gradient method with unit steplength on the following convex problem
\begin{align} \label{eq:dgd_pen}
\min_{\textbf{x} \in \mathbb{R}^{np}} \frac{1}{2}\textbf{x}^T({I- \textbf{Z}})\textbf{x} + \alpha\sum_{i=1}^{n}f_i(x_i).
\end{align}

The theoretical properties of the DGD method have been well established; see \cite{bertsekas1989parallel,nedic2009distributed,yuan2016convergence}. The convergence results are typically established under the following standard assumptions:
\begin{assum}\label{assm:Lip}
	 Each local objective function $f_i$  has $L_i$-Lipschitz continuous gradients.
\end{assum}
\begin{assum}\label{assm:Strong_h}
	{The objective function $h$ \eqref{eq:prob} is $\mu_h$-strongly convex.}
\end{assum}
\begin{assum}\label{assm:Strong}
	Each local objective function $f_i$ is $\mu_i$-strongly convex. {Note, this Assumption implies Assumption \ref{assm:Strong_h}.}
\end{assum}

{These assumptions guarantee the existence of a unique optimal solution. Under Assumption \ref{assm:Lip}, the DGD method can be shown to converge  at a sublinear rate with diminishing stepsize (decrease stepsize $\alpha$ as the algorithm evolves). The diminishing stepsize is effectively shrinking the penalty parameter $\alpha$ of problem \eqref{eq:dgd_pen} and thus DGD recovers a feasible and optimal solution of problem \eqref{eq:cons_prob1} in the limit. If we further assume the conditions in Assumption \ref{assm:Strong}, then with an appropriate constant stepsize, DGD converges at a linear rate to the optimal solution of \eqref{eq:dgd_pen}, which is in a neighborhood of the optimal solution of \eqref{eq:cons_prob1} \cite{yuan2016convergence}. The limit point of DGD with constant stepsize is often infeasible for the equality constraint in problem \eqref{eq:cons_prob1}.}


\medskip
\noindent
\textbf{Notation:} For the rest of the paper, we follow the same notation as in this section. A boldface lower case letter indicates a concatenated vector, i.e.,  $\textbf{v}\in \mathbb{R}^{np}$ represents the
concatenation of local vectors $v_i \in \mathbb{R}^p$. Notation $\bar{v} \in \mathbb{R}^p$ denotes the average of all local vectors $v_i \in \mathbb{R}^p$, i.e., $\bar{v} = \frac{1}{n} \sum_{i=1}^n v_i$. The two subscripts $x_{i,k}$ indicate the agent index $i$ and iteration count $k$.

\section{DGD$^t$: A Distributed Gradient Descent Variant}
\label{sec:DGDt}

A close inspection of the DGD iterate update Eq.\ \eqref{eq:dgd_vec} reveals that the method performs a single round of communication and a single computation per iteration.
However, a natural question is whether this is optimal or even necessary. Restating this question from a different angle: is there flexibility in creating a whole class of methods based on the components of the DGD method that perform different number of communication and computation steps \ew{depending on the application}? Motivated by this question and our adaptive cost framework (Section \ref{sec:cost}), we decompose and rearrange the communication and computation steps of DGD to  construct more flexible algorithms. We present two improved classes of DGD-based algorithms in this and the following sections, which we call DGD$^t$ and \ndgd methods, respectively.  We also provide answers to the following questions: (i) what the interpretation of these new methods are, and (ii) what theoretical guarantees can be established. For simplicity, we will focus on the constant stepsize implementations.

\ew{ For the first class of algorithms, we consider scenarios in which communication is much cheaper than computation, as in the shared memory machine learning example. We note that a major drawback of the DGD algorithm is that to obtain a feasible solution we need to use diminishing stepsizes, which results in slow convergence speed. With constant stepsizes, the resulting solution of DGD is infeasible with respect to the equality constraint of problem \eqref{eq:cons_prob1}. In order to improve the solution quality without sacrificing convergence speed, we propose to perform $t$ consensus steps at each iteration, and} consider the following constant stepsize iterate update equation,
\begin{align}	\label{eq:dgdt_vec}
	\textbf{x}_{k+1} = \textbf{Z}^{t}\textbf{x}_k - \alpha \nabla \textbf{f}(\textbf{x}_k), \qquad \textbf{Z}^{{t}} = \textbf{W}^{{t}}\otimes I_p,
\end{align}
which we call the {\it DGD$^t$} method. The DGD$^t$ method can be thought of as a gradient method with unit steplength on the following convex problem
\begin{align} \label{eq:dgd_penalty}
\min_{\textbf{x} \in \mathbb{R}^{np}} p_f(\textbf{x}) = \frac{1}{2}\textbf{x}^T({I- \textbf{Z}^t})\textbf{x} + \alpha\sum_{i=1}^{n}f_i(x_i).
\end{align}

The intuition behind this method is that by increasing the number of consensus steps from $1$ to $t$ per iteration, the resulting solution should be closer to being feasible. Alternatively, we can view the DGD$^t$ method as a DGD method with a different underlying graph (different weights in $\textbf{W}$). \ew{ As we will show  next, the solution of DGD$^t$ is indeed closer to being feasible compared to standard DGD. This is achieved at the cost of more communication steps per iteration. Also, unlike DGD, where the gradient computation and consensus can happen simultaneously, the $t$ communication steps in DGD$^t$ have to happen sequentially. This method can be desirable when communication is cheap, i.e., $c_c$ is much smaller than $c_g$. }

\subsection{Convergence Analysis of DGD$^t$}

We now provide a complete convergence analysis for the DGD$^t$ method with constant stepsize. {We should note again that DGD$^t$ is a variant of DGD by replacing the weight matrix $W$ by $W^t$. As such, our analysis follows a similar approach as in \cite{yuan2016convergence}, and so for brevity we have omitted the proofs. The proofs can be found in \cite{balCC}.} 

For notational convenience, we introduce the following quantities that are used in the analysis
\begin{gather}		\label{eq:g_barg}
	\bar{x}_k = \frac{1}{n}\sum_{i=1}^n x_{i,k}, \quad g_k =  \frac{1}{n}\sum_{i=1}^n \nabla f_i(x_{i,k}),  \nonumber\\
	 \bar{g}_k =  \frac{1}{n}\sum_{i=1}^n \nabla f_i(\bar{x}_{k}).
\end{gather}
Vector  $\bar x_k \in \mathbb{R}^p$ corresponds to the average of local estimates, vector  $g_k \in\mathbb{R}^p$ represents the average of local gradients at the current local estimates and vector $\bar g_k \in\mathbb{R}^p$ indicates the average gradient at $\bar x_k$.

\begin{lem} 	\label{lem:dgdt_bound_grads}
\textbf{(Bounded gradients)} Suppose Assumption \ref{assm:Lip} holds, and let the steplength satisfy
\begin{align}	\label{eq:lem1_alpha}
	\alpha \leq {\frac{1+\lambda_n(\textbf{W}^{{t}})}{L}}
\end{align}
where ${\lambda_n(\textbf{W}^{{t}})}$ is the smallest eigenvalue of ${\textbf{W}^{{t}}}$ and $L = \max_{i}L_i$. Then, starting from $x_{i,0} = 0$ ($1\leq i \leq n$), the sequence $x_{i,k}$ generated by the DGD$^t$ method converges. In addition, we also have
\begin{align*}	
	\| \nabla \textbf{f} (\textbf{x}_k) \| \leq D = \sqrt{2L \left( \sum_{i=1}^n \left(f_i(0) - f_i^\star \right)\right)}
\end{align*}
for all $k = 1,2,\ldots$, where $f_i^\star = f_i(x_i^\star)$ and $x_i^\star = \arg\min_x f_i(x)$.
\end{lem}

\begin{proof} 
Note that the DGD$^t$ iteration \eqref{eq:dgdt_vec} is equivalent to a gradient descent iteration, with unit steplength on the quadratic penalty function $p_f$ \eqref{eq:dgd_penalty}. We first show that the function $p_f$ has $ \left[\left(1 - \lambda_n(\textbf{W}^t) \right) + \alpha L \right]$-Lipschitz continuous gradients. By definition of $p_f$ and the triangle inequality, we have
\begin{align} \label{ineq:Lips}
\norm{\nabla p_f(u)- \nabla p_f(v)} &\leq \norm{(I-\textbf{Z}^t)(u-v)} \nonumber\\
& \quad + \alpha \norm{\sum_{i=1}^{n}\nabla f_i(u_i) - \sum_{i=1}^{n}\nabla f_i(v_i)}.
\end{align}
By the Cauchy-Schwartz inequality, the first term satisfies
\begin{align*}
\norm{(I-\textbf{Z}^t)(u-v)}&\leq \norm{I-\textbf{Z}^t}\norm{u-v} \\
& = \left(1-\lambda_n(\textbf{W}^t)\right)\norm{u-v},
\end{align*}
where the last inequality follows from the fact that all eigenvalues of the matrix $\textbf{W}$ lie in the interval $(-1, 1]$. The second term in Eq.\ \eqref{ineq:Lips} satisfies
\[ \alpha \norm{\sum_{i=1}^{n}\nabla f_i(u_i) - \sum_{i=1}^{n}\nabla f_i(v_i)}\leq \alpha L \norm{u-v},\] due to Assumption \ref{assm:Lip}. Thus, we have that the function $p_f$ has Lipschitz continuous gradients with
\begin{align*}
	L_{p_f} \leq \left(1 - \lambda_n(\textbf{W}^t) \right) + \alpha L.
\end{align*}
From the classical analysis of gradient descent \cite{bertsekas1989parallel}, we know that these iterates will converge with unit stepsize if $1 \leq \frac{2}{L_{p_f}}$, where $L_{p_f}$ is the Lipschitz constant of the gradients of $p_f$.
Since $\alpha \leq {\frac{1+\lambda_n(\textbf{W}^{{t}})}{L}}$ from Eq.\ \eqref{eq:lem1_alpha}, we have
\begin{align*}
 L_{p_f} &\leq \left(1 - \lambda_n(\textbf{W}^t) \right) + \alpha L\\
 &\leq \left(1 - \lambda_n(\textbf{W}^t) \right) + {\frac{1+\lambda_n(\textbf{W}^{{t}})}{L}} L\leq 2.
 \end{align*}
Hence when the condition in Eq.\ \eqref{eq:lem1_alpha} is satisfied, the iterates $\textbf{x}_k$ will converge which implies that the individual iterates $x_{i,k}$ converge.

We now show the bound on the gradients. Since the function values obtained in the gradient descent method are non-increasing and $I - \textbf{W}^t$ is a positive semi-definite matrix, we have,
\begin{align}	\label{eq:l1_1}
	 \sum_{i=1}^{n} f_i(x_{i,k}) \leq \frac{1}{\alpha} p_f(\textbf{x}_k) \leq \frac{1}{\alpha} p_f(\textbf{x}_{k-1}) \leq \dots \nonumber \\
	 \dots \leq \frac{1}{\alpha}p_f(\textbf{x}_{0}) = \sum_{i=1}^{n} f_i(0).
\end{align}

By Theorem 2.1.5 in \cite{nesterov2013introductory}, any convex function $\phi$ with $L-$Lipschitz gradient satisfies
\[\phi(x)+ \nabla \phi (x)^T(y-x)+\frac {1}{2L}\norm{\nabla \phi(x) - \nabla \phi(y)}^2\leq \phi (y)\]
for all $x, y$ in its domain. We apply this relation to each of $f_i$ at respective $x_i^\star$, and have
\begin{align}	\label{eq:l1_2}
f_i(x_i^\star)+ \frac {1}{2L_i}\norm{\nabla f_i(x)}^2\leq f_i (x)
\end{align}
for all $x$ in $\mathbb{R}^p$.

Finally, we can bound $	\| \nabla \textbf{f}(\textbf{x}_k) \|^2$ by 
 \begin{align*}
 	\| \nabla \textbf{f}(\textbf{x}_k) \|^2 = \sum_{i=1}^{n}\|\nabla f_i(x_{i,k})\|^2 &\leq \sum_{i=1}^{n} 2L_i \left(f_i(x_{i,k}) - f^\star_i \right) \\
	&\leq 2L\left( \sum_{i=1}^{n}\left(f_i(0) - f^\star_i \right)\right).
 \end{align*}
where the first inequality follows from Eq. \eqref{eq:l1_2} and the second inequality uses the definition of $L$ and Eq. \eqref{eq:l1_1}.
\end{proof}

Lemma \ref{lem:dgdt_bound_grads} shows that the iterates produced by the DGD$^t$ method converge and have bounded gradients. A different bound can be shown if $x_{i,0} \neq 0$ for all $i$.  For convenience, we assume that $x_{i,0} = 0$, for all $i$ for the rest of this section.

\begin{lem} 	\label{lem:dgdt_bound_dev_mean}
\textbf{(Bounded deviation from mean)} If Assumptions \ref{assm:Lip}-\ref{assm:Strong_h} hold. Then, starting from $x_{i,0} = 0$, the total deviation from the mean is bounded, namely,
\begin{align*}
	{\| x_{i,k} - \bar{x}_k\| \leq \frac{\alpha D}{1 - \beta^{{t}}}},
\end{align*}
and
\begin{gather}
	{\| \nabla f_i(x_{i,k}) - \nabla f_i(\bar{x}_k)\| \leq \frac{\alpha D L_i}{1 - \beta^{{t}}}} , \label{eq:lem2_p2}\\
	{\| g_k - \bar{g}_k\| \leq \frac{\alpha D L}{1 - \beta^{{t}}}}, \label{eq:lem2_p3}
\end{gather}
for all $k=1,2,\ldots$ and $1 \leq i \leq n$.
\end{lem}

\begin{proof} 
By iteratively applying the DGD$^t$ iteration \eqref{eq:dgdt_vec} and the definition of $\textbf{x}$, we obtain
\begin{align*}		
	\textbf{x}_k = -\alpha \sum_{s=0}^{k-1}\left(\textbf{W}^{t(k-1-s)} \otimes I \right) \nabla \textbf{f}(x_s).
\end{align*}
Let $\bar{\textbf{x}}_k = [\bar{x}_k;\bar{x}_k;\ldots;\bar{x}_k] \in \mathbb{R}^{np}$, it follows that
\begin{align*}
	\bar{\textbf{x}}_k = \frac{1}{n} \left((1_n 1_n^T ) \otimes I \right) \textbf{x}_k.
\end{align*}
As a result,
\begin{align*}
	&\|x_{i,k} - \bar{x}_k\| \leq \|\textbf{x}_k - \bar{\textbf{x}}_k\| \\
	&=\left\|\textbf{x}_k - \frac{1}{n}\left((1_n 1_n^T) \otimes I\right) \textbf{x}_k\right\| \\
	&= \left\| -\alpha \sum_{s=0}^{k-1}\left(\textbf{W}^{t(k-1-s)}\otimes I\right) \nabla \textbf{f}(x_s)  \right.\\&\qquad+\left. \alpha \sum_{s=0}^{k-1} \frac{1}{n}\left((1_n 1_n^T\textbf{W}^{t(k-1-s)}) \otimes I \right) \nabla \textbf{f}(x_s)\right\| \\
	&= \left\| -\alpha \sum_{s=0}^{k-1}\left(\textbf{W}^{t(k-1-s)}\otimes I \right) \nabla \textbf{f}(x_s)  \right.\\
	&\qquad+\left. \alpha \sum_{s=0}^{k-1} \frac{1}{n}\left((1_n 1_n^T) \otimes I \right) \nabla \textbf{f}(x_s)\right\|, \\
\end{align*}
where the third equality holds since $\textbf{W}^t$ is doubly-stochastic, which is a direct consequence of $\textbf{W}$ being doubly-stochastic. Thus we have, 
\begin{align*}
	\|x_{i,k} - \bar{x}_k\| &\leq \left\| -\alpha \sum_{s=0}^{k-1}\left(\textbf{W}^{t(k-1-s)}\otimes I \right) \nabla \textbf{f}(x_s)  \right.\\
	&\qquad  \qquad+\left. \alpha \sum_{s=0}^{k-1} \frac{1}{n}\left((1_n 1_n^T) \otimes I \right) \nabla \textbf{f}(x_s)\right\| \\
	&=\alpha \left\| \sum_{s=0}^{k-1}\left( (\textbf{W}^{t(k-1-s)} - \frac{1}{n}1_n 1_n^T ) \otimes I \right) \nabla \textbf{f}(x_s)\right\| \\
	&\leq \alpha \sum_{s=0}^{k-1} \left\|\textbf{W}^{t(k-1-s)} - \frac{1}{n}1_n1_n^T\right\|\|\nabla \textbf{f}(x_s)\| \\
	&= \alpha \sum_{s=0}^{k-1} \beta^{t(k-1-s)}\|\nabla \textbf{f}(x_s)\|,
\end{align*}
where the inequality is due to Cauchy-Schwartz, and the last equality follows from the definition of $\beta$, since the matrix $\frac{1}{n}1_n1_n^T$ is the projection of $\textbf{W}$ onto the eigenspace associated with the eigenvalue equal to $1$. Using Lemma \ref{lem:dgdt_bound_grads} and the fact that $\beta<1$, it follows that
\begin{align}	\label{eq:lem2_p1}
	\|x_{i,k} - \bar{x}_k\| &\leq \alpha \sum_{s=0}^{k-1} \beta^{t(k-1-s)}\|\nabla \textbf{f}(x_s)\| \nonumber\\
	&\leq \alpha D \sum_{s=0}^{k-1} \beta^{t(k-1-s)} \leq \frac{\alpha D}{1 - \beta^t}.
\end{align}
%

The result \eqref{eq:lem2_p2} is a direct consequence of \eqref{eq:lem2_p1} and the Lipschitz continuity of the individual gradients (Assumption \ref{assm:Lip}). For the second result \eqref{eq:lem2_p3}, we have
\begin{align*}
	\|g_k - \bar{g}_k\|  &= \left\|\frac{1}{n} \sum_{i=1}^{n} \left(\nabla f_i(x_{i,k}) -  \nabla f_i(\bar{x}_k)\right) \right\| \\
&\leq \frac{1}{n} \sum_{i=1}^{n} L_i \|x_{i,k} - \bar{x}_k\| \leq \frac{\alpha D L}{1 - \beta^t}.
\end{align*}
\end{proof}

Lemma \ref{lem:dgdt_bound_dev_mean} shows that the distance between the local iterates and the average is bounded. As a consequence of this, the deviation in the gradients is also bounded.

We now look at the optimization error. We observe that due to the doubly-stochastic nature of $\textbf{W}$,
\begin{align*}
	\bar{\textbf{x}}_{k+1} &= \frac{1}{n}\left((1_n1_n^T) \otimes I \right){\textbf{x}}_{k+1}\\
	& = \frac{1}{n}\left((1_n1_n^T) \otimes I \right)\left((\textbf{W}^{{t}} \otimes I)\textbf{x}_{k} - \alpha \nabla \textbf{f}(\textbf{x}_k) \right)\\
	& = \frac{1}{n}\left((1_n1_n^T\textbf{W}^{{t}}) \otimes I \right)\textbf{x}_{k} -  \frac{\alpha}{n} \left((1_n1_n^T) \otimes I \right) \nabla \textbf{f}(\textbf{x}_k)\\
	& = [\bar{x}_k - \alpha g_k; \bar{x}_k - \alpha g_k;\ldots;\bar{x}_k - \alpha g_k].
\end{align*}
Thus we have
\begin{align}	\label{eq:dgdt_errors}
	\bar{{x}}_{k+1} = \bar{{x}}_k - \alpha g_k.
\end{align}
Recall that $g_k$ is the average of gradients at the current local estimates [cf.\ Eq.\ \eqref{eq:g_barg}] and thus the above equation can be viewed as an inexact gradient descent step for the problem
\begin{align}		\label{eq:prob_bar}
	\min_{x\in \mathbb{R}^p} \bar{f}(x) = \frac{1}{n} \sum_{i=1}^n f_i (x),
\end{align}
where $\bar{g}_k$ is the exact gradient (at the average of the local estimates).
Consequently, if $h$ has $L_h$-Lipschitz continuous gradients, and is $\mu_h$-strongly convex, then it can be shown that the function $\bar{f}$ has $L_{\bar{f}}$-Lipschitz continuous gradients and is $\mu_{\bar{f}}$ strongly convex with
\begin{align*}
	L_{\bar{f}} = \frac{1}{n} \sum_{i=1}^n L_i = \frac{1}{n} L_h, \quad
	\mu_{\bar{f}}  = \frac{1}{n} \sum_{i=1}^n \mu_i = \frac{1}{n} \mu_h.
\end{align*}
Based on the above observations, we bound the distance to the optimal solution.

\begin{thm}		\label{thm:dgdt_bound_dist_min}
\textbf{(Bounded distance to optimal solution)} Suppose Assumptions \ref{assm:Lip}-\ref{assm:Strong_h} hold, and let the steplength satisfy
\begin{align*}
	\alpha \leq \min \left \{ {\frac{1+\lambda_n(\textbf{W}^{{t}})}{{L}}, c_4} \right \}
\end{align*}
where ${\lambda_n(\textbf{W}^{{t}})}$ is the smallest eigenvalue of ${\textbf{W}^{{t}}}$, $L = \max_{i}L_i$ and $c_4 = \frac{{2}}{\mu_{\bar{f}} +L_{\bar{f}}}$. Then, starting from $x_{i,0} = 0$ ($1\leq i \leq n$), for all $k=0,1,2,\ldots$
\begin{align*}
	\| \bar{x}_{k+1} - x^\star\|^2 \leq c_1^2 \| \bar{x}_{k} - x^\star \|^2 + \frac{{{c_3^2}}}{{{(1-\beta^{{t}})^2}}},
\end{align*}
where
\begin{align*}
	c_1^2 = 1 - \alpha c_2 + \alpha \delta - \alpha^2 \delta c_2, \;\; c_2 = \frac{{2}\mu_{\bar{f}} L_{\bar{f}}}{\mu_{\bar{f}} + L_{\bar{f}}},\\
	c_3^2 = \alpha^3(\alpha + \delta^{-1})L^2 D^2, \;\; D = \sqrt{2L \left(\sum_{i=1}^n f_i(0) - f^\star \right)},
\end{align*}
 $x^\star$ is the optimal solution of \eqref{eq:cons_prob1} and $\delta>0$. In particular, if we set $\delta = \frac{c_2}{2(1-\alpha c_2)}$ such that $c_1 = \sqrt{1 - \frac{\alpha c_2}{2}} \in (0,1)$, then for $k=0,1,2,\ldots$
 \begin{align*}
 	\| \bar{x}_{k} - x^\star\| \leq c_1^k \| \bar{x}_{0} - x^\star\| + \mathcal{O}\left( \frac{\alpha}{1-\beta^t} \right).
 \end{align*}
\end{thm}

\begin{proof} 
Using the definitions of the $\bar{x}_k$, $g_k$ and \eqref{eq:dgdt_errors}, we have
\begin{align}		
	\| \bar{x}_{k+1} - x^\star \|^2 &= \| \bar{x}_{k} - x^\star - \alpha g_k\|^2 \nonumber\\
		& = \| \bar{x}_{k} - x^\star - \alpha \bar{g}_k + \alpha (\bar{g}_k - g_k)\|^2 \nonumber\\
		& = \| \bar{x}_{k} - x^\star - \alpha \bar{g}_k \|^2 + \alpha^2 \| \bar{g}_k - {g}_k \|^2 \nonumber\\
		& \quad+ 2\alpha (\bar{g}_k - g_k)^T(\bar{x}_{k} - x^\star - \alpha \bar{g}_k)\nonumber\\
		& \leq (1 + \alpha \delta)\| \bar{x}_{k} - x^\star - \alpha \bar{g}_k\|^2 \nonumber\\
		& \quad+ \alpha(\alpha + \delta^{-1})\| \bar{g}_k - {g}_k \|^2, \label{eq:dgdt_thm1}
\end{align}
where the last inequality follows from the fact that for any vectors $a$ and $b$, $\pm 2a^Tb \leq \delta^{-1} \| a\|^2 + \delta \| b \|^2$, for $\delta>0$.

We now bound the quantity $\| \bar{x}_{k} - x^\star - \alpha \bar{g}_k\|^2$,
\begin{align}		\label{eq:dgdt_thm2}
	\| &\bar{x}_{k} - x^\star - \alpha \bar{g}_k\|^2 \nonumber \\ &= \|  \bar{x}_{k} - x^\star \|^2 + \alpha^2 \|  \bar{g}_k\|^2 - 2(\bar{x}_{k} - x^\star)^T(\alpha \bar{g}_k) \nonumber \\
	& \leq \|  \bar{x}_{k} - x^\star \|^2 + \alpha^2 \| \bar{g}_k\|^2 - \alpha c_4 \| \bar{g}_k\|^2 - \alpha c_2 \| \bar{x}_{k} - x^\star \|^2 \nonumber \\
	& = (1 - \alpha c_2)\|  \bar{x}_{k} - x^\star \|^2 + \alpha(\alpha -  c_4)\| \bar{g}_k\|^2 \nonumber\\
	& \leq (1 - \alpha c_2)\|  \bar{x}_{k} - x^\star \|^2,
\end{align}
where the first inequality follows from \cite[Theorem 2.1.12, Chapter 2]{nesterov2013introductory}, and in the last inequality we dropped the term $\alpha(\alpha -  c_4)\| \bar{g}_k\|^2$, since $\alpha\leq c_4$ and $\alpha(\alpha -  c_4)\| \bar{g}_k\|^2 \leq 0$.
%
By combining \eqref{eq:dgdt_thm1} and \eqref{eq:dgdt_thm2}, and using \eqref{eq:lem2_p3}, we obtain
\begin{align}		\label{eq:dgdt_onestep}	
	\| \bar{x}_{k+1} - x^\star \|^2 & \leq (1 + \alpha \delta)(1 - \alpha c_2)\|  \bar{x}_{k} - x^\star \|^2 \nonumber\\
	& \quad+ \alpha^3(\alpha + \delta^{-1}) \frac{ L^2 D^2 }{(1 - \beta^{{t}})^2}.
\end{align}

Combining  $c_4 < \frac{1}{c_2}$, and $\alpha\leq c_4$, we have $(1 + \alpha \delta)(1 - \alpha c_2)>0$. Now, using the definitions of $c_1$ and $c_3$, and by recursive application of \eqref{eq:dgdt_onestep}, we have
Now, using the definitions of $c_1$ and $c_3$, and by recursive application of \eqref{eq:dgdt_onestep}, we have
\begin{align*}
	\| \bar{x}_{k} - x^\star \|^2 &\leq c_1^{2k} \| \bar{x}_{0} - x^\star\|^2 + \frac{c_3^2}{(1-c_1^2)(1-\beta^{{t}})^2},
\end{align*}
and so
\begin{align*}
	\| \bar{x}_{k} - x^\star \| &\leq c_1^{k} \| \bar{x}_{0} - x^\star\| + \frac{c_3}{\sqrt{1-c_1^2}(1-\beta^{{t}})}.
\end{align*}
If $\delta = \frac{c_2}{2(1-\alpha c_2)}>0$, then
\begin{align*}
	c_1^2 = 1 - \frac{\alpha c_2}{2}<1
\end{align*}
and
\begin{align*}
	\frac{c_3}{\sqrt{1-c_1^2}(1-\beta^{{t}})} = \mathcal{O}\left( \frac{\alpha}{1-\beta^t} \right).
\end{align*}
which completes the proof.
\end{proof}

Theorem \ref{thm:dgdt_bound_dist_min} shows that the average of the iterates generated by the DGD$^t$ method converges to a neighborhood of the solution that is defined by the steplength, the second largest eigenvalue of $\textbf{W}$ and the number of consensus steps.

\begin{cor}		\label{cor:dgdt_bound_dist_min}
\textbf{(Local agent convergence)} Suppose Assumptions \ref{assm:Lip}-\ref{assm:Strong_h} hold, and let the steplength satisfy
\begin{align*}
		\alpha \leq \min \left \{ {\frac{1+\lambda_n(\textbf{W}^{{t}})}{L}, c_4} \right \}.
\end{align*}
where $L = \max_{i}L_i$ and $c_4 = \frac{{2}}{\mu_{\bar{f}} +L_{\bar{f}}}$. Then, starting from $x_{i,0} = 0$ ($1\leq i \leq n$), for $k=0,1,2,\ldots$
\begin{align*}
	\| x_{i,k} - x^\star \| \leq c_1^k \| x^\star\| + \frac{c_3}{\sqrt{1-c_1^2}(1-\beta^t)} + \frac{\alpha D}{1-\beta^t}.
\end{align*}
\end{cor}

\begin{proof} 
Using the results from Lemma \ref{lem:dgdt_bound_dev_mean}, Theorem \ref{thm:dgdt_bound_dist_min} and the definition of $c_4$, we have
\begin{align*}
	\| x_{i,k} - x^\star \| &\leq \| \bar{x}_k - x^\star \| + \| x_{i,k} - \bar{x}_k \| \\
		& \leq c_1^k \| x^\star\| + \frac{c_3}{\sqrt{1-c_1^2}(1-\beta^t)} + \frac{\alpha D}{1-\beta^t}.
\end{align*}
\end{proof}

The results of Theorem \ref{thm:dgdt_bound_dist_min} and Corollary \ref{cor:dgdt_bound_dist_min} are similar in nature to the results for the standard DGD method. More specifically, and not surprisingly, the DGD$^t$ method converges to a neighborhood of the optimal solution of problem \eqref{eq:cons_prob1} when a constant steplength is employed. Compared to the DGD method, DGD$^t$ converges to a better (smaller) neighborhood which is captured in the analysis
\begin{align*}
		\mathcal{O}\left(\frac{1}{(1-\beta)^2}\right) \quad \text{v.s.} \quad \mathcal{O}\left(\frac{1}{(1-\beta^{{t}})^2}\right).
\end{align*}
Performing multiple communication steps is beneficial as it improves the neighborhood of convergence, however, multiple consensus alone cannot guarantee convergence of the DGD$^t$ method to the solution. Namely, the error term that appears in Theorem \ref{thm:dgdt_bound_dist_min} cannot be eliminated by simply performing multiple rounds of communication, since $\lim_{t \rightarrow \infty} \frac{1}{(1-\beta^{{t}})^2} = 1 \neq 0 $. 

The results presented in Lemmas \ref{lem:dgdt_bound_grads} and \ref{lem:dgdt_bound_dev_mean}, Theorem \ref{thm:dgdt_bound_dist_min} and Corollary \ref{cor:dgdt_bound_dist_min} are not surprising, nevertheless, the results clearly illustrate the power of performing multiple rounds of communication steps.

\subsection{Numerical Results for DGD$^t$}
\label{sec:num_res_dgdt}


In this section, we provide some empirical evidence to support the theoretical results, and to ascertain that, in practice, the benefits of performing multiple consensus steps are realized. We chose a simple quadratic problem of the form
\begin{align*}
		 f(x) = \frac{1}{2} \sum_{i=1}^n x^T A_i x + b_i^Tx,
\end{align*}
where each node $i=\{1,...,n\}$ has local information $A_i \in \mathbb{R}^{p\times p}$ and $b_i \in \mathbb{R}^{p} $. The problem was constructed as described in \cite{mokhtari2017network}; we chose a dimension size $p=10$ and the condition number ($\kappa =  \frac{L_f}{\mu_f}$) was set to $10^2$. For this experiment, the number of agents in the network ($n$) was $10$, and we considered a $4$-cyclic graph topology (i.e., each node is connected to its $4$ immediate neighbors).

We investigated the performance of four different variants of DGD$^t$, with $t=1$ (standard DGD) and $t=2,5,10$. Figure \ref{multiple_cons_DGD} illustrates the performance (relative error, $\|\bar{x}_k - x^\star\|^2/\| x^\star \|^2$, with $x^\star \neq 0$) of the four methods in terms of: (i) iterations, (ii) cost (as described in Section \ref{sec:cost}, with $c_c = c_g = 1$), (iii) number of gradient evaluations, and (iv) number of communications. Each of the four methods has a steplength parameter that needs to be tuned in order to achieve the best performance. We manually tuned the steplength for each method independently. We mention in passing that similar behavior was observed when we measured the performance of the methods in terms of consensus error $\left(\frac{1}{n}\sum_{i=1}^n\|{x}_{i,k} - x^\star\|^2/\| x^\star \|^2\right)$.

\begin{figure}[]
\centering
\includegraphics[width=0.41\columnwidth]{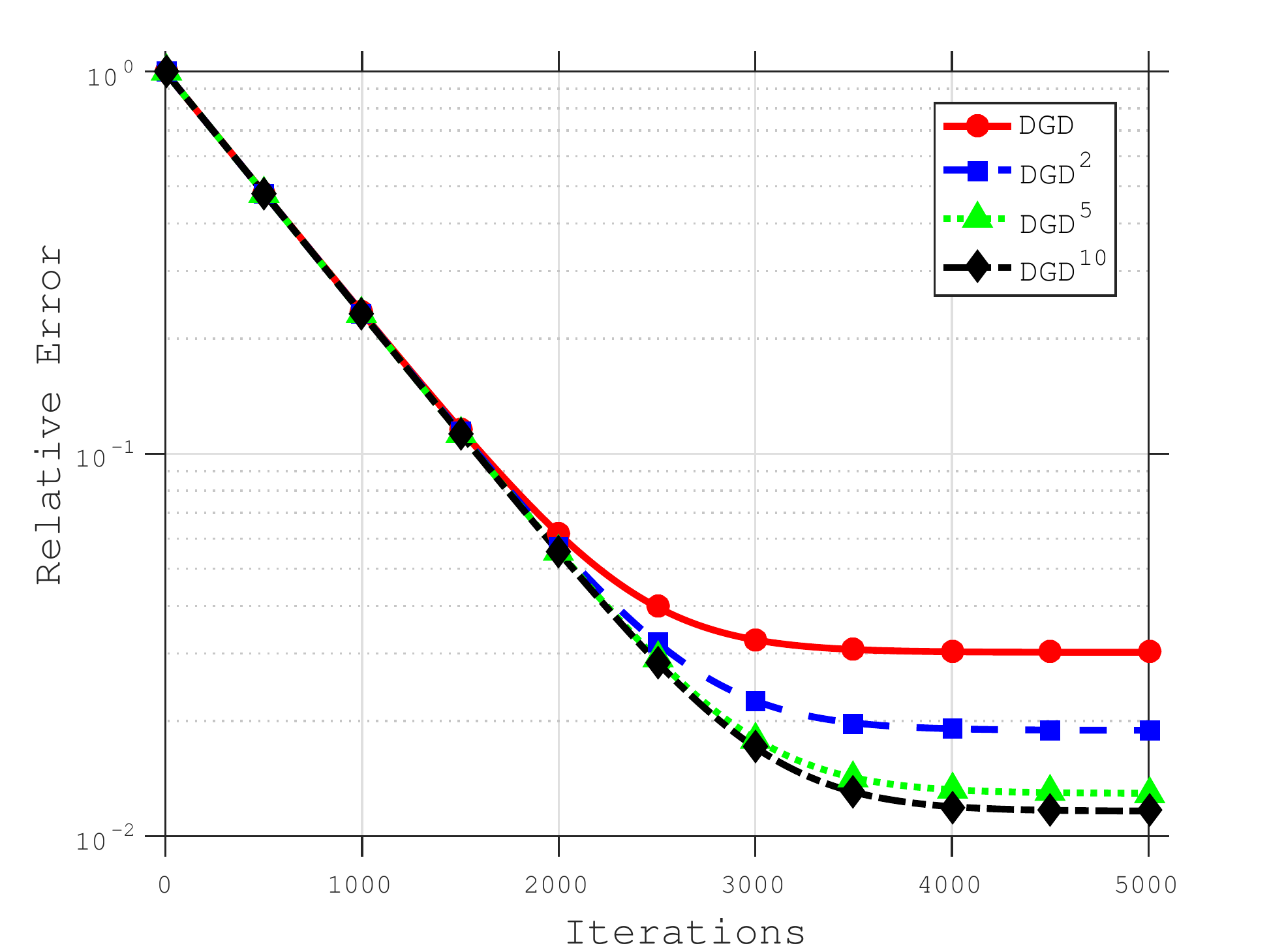}
\includegraphics[width=0.41\columnwidth]{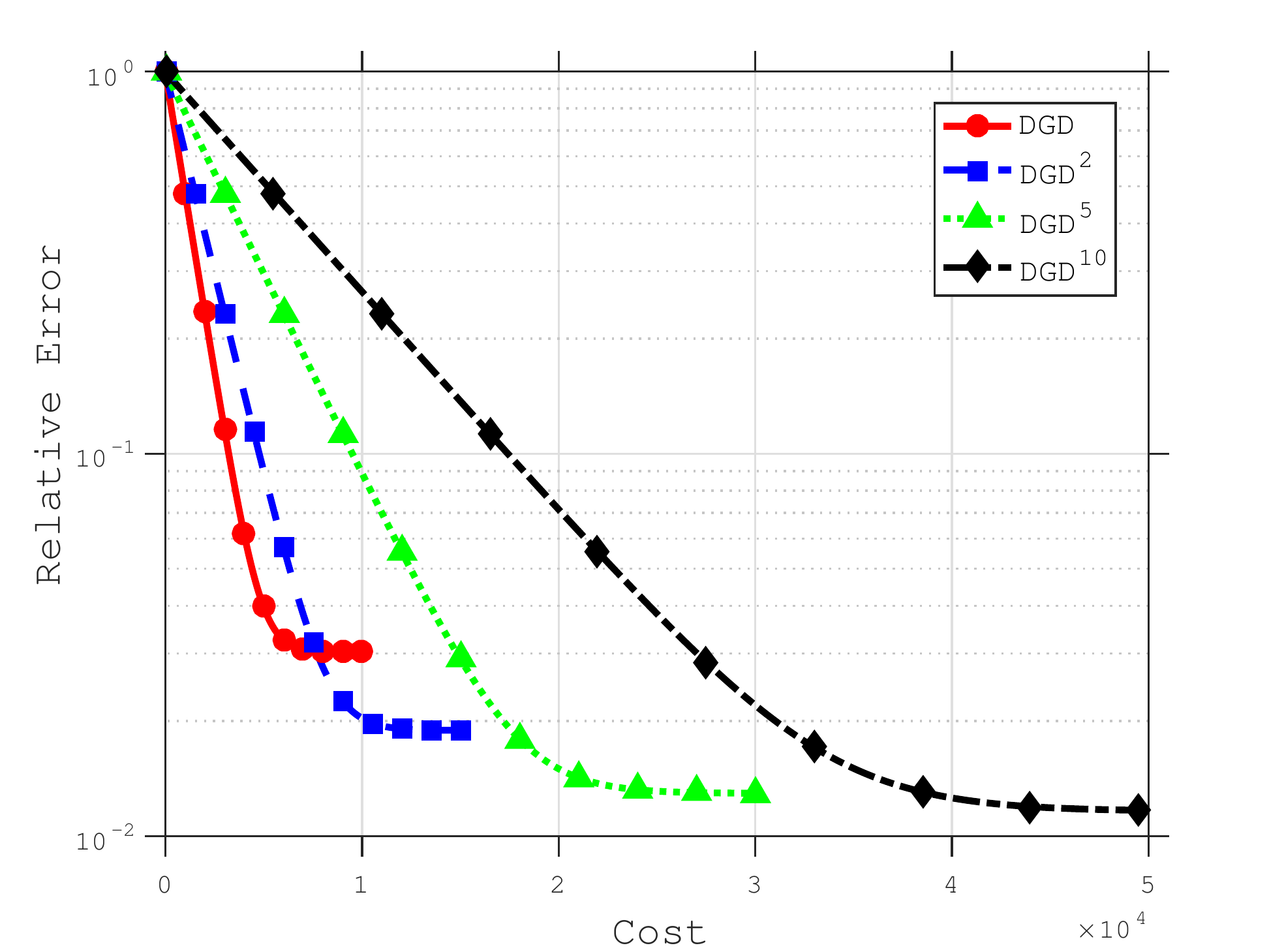}
\hspace{0.1cm}
\includegraphics[width=0.41\columnwidth]{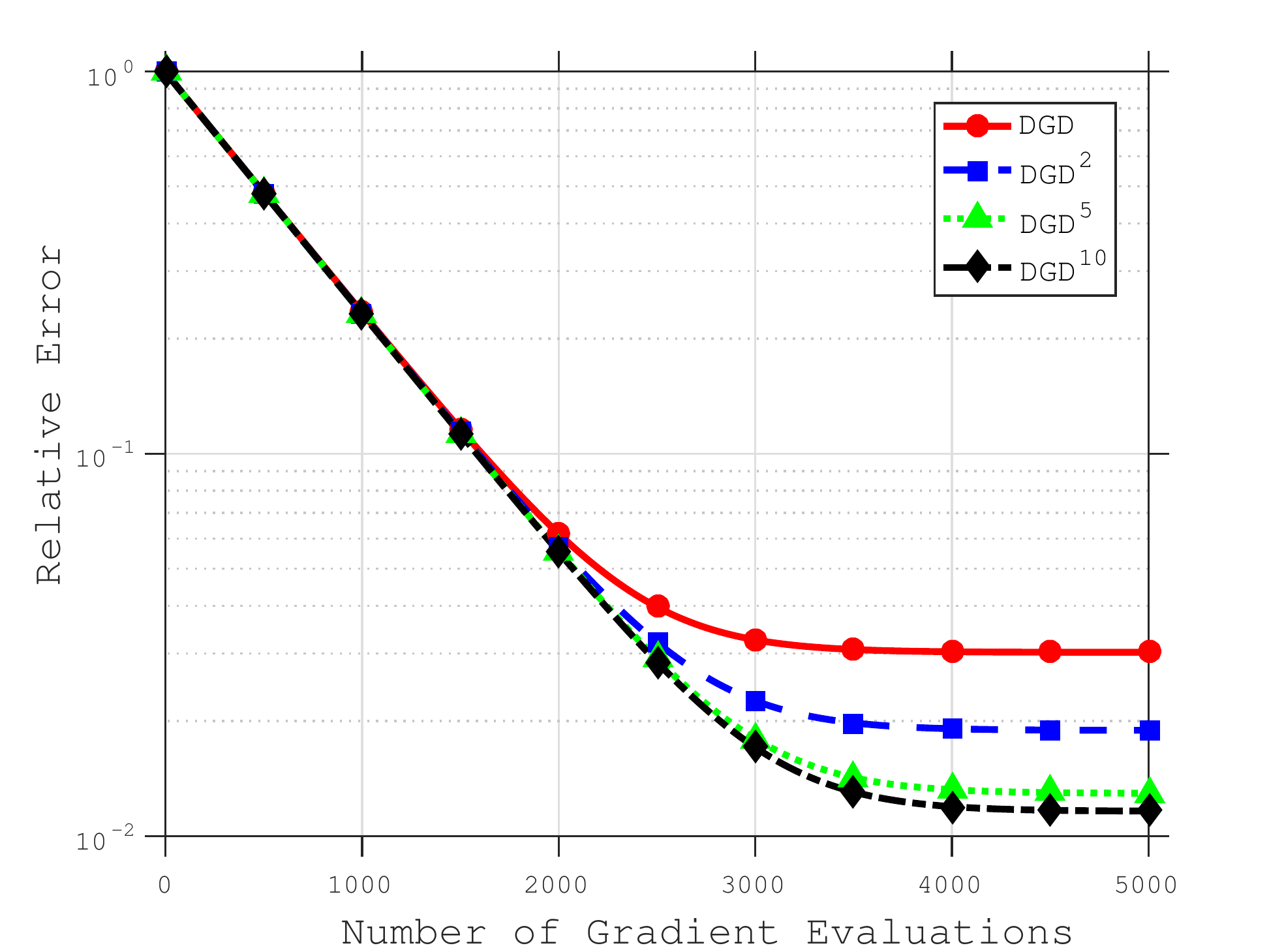}
\includegraphics[width=0.41\columnwidth]{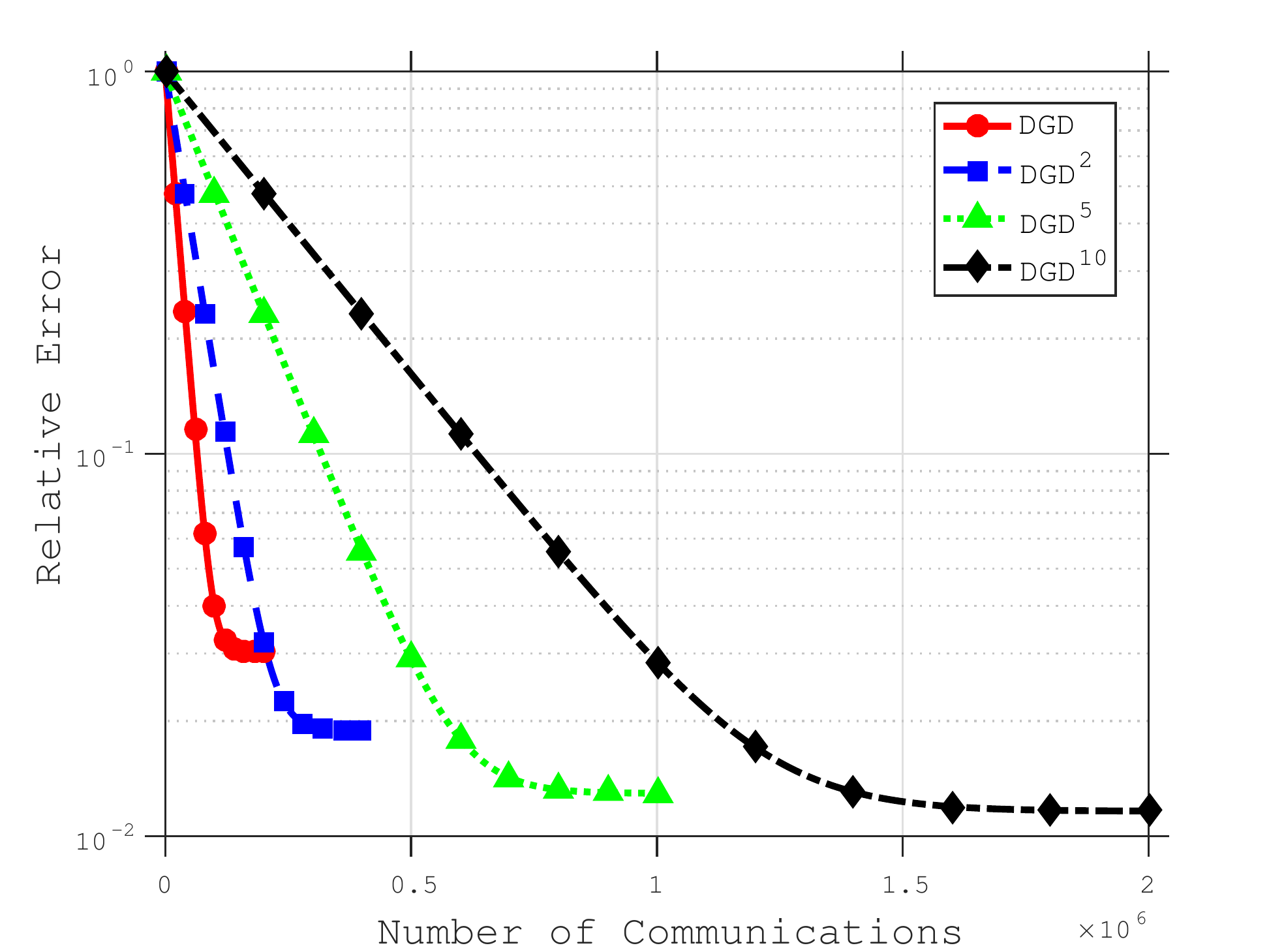}
\caption{Performance of DGD, DGD$^2$, DGD$^5$ and DGD$^{10}$ measured in terms of relative error $\left(\|\bar{x}_k - x^\star\|^2/\| x^\star \|^2\right)$ with respect to: (i) number of iterations, (ii) cost, (iii) number of gradient evaluations, and (iv) number of communications, on a quadratic problem ($n = 10$, $p=10$, $\kappa = 10^2$).}
\label{multiple_cons_DGD}
\end{figure}

Figure \ref{multiple_cons_DGD} clearly illustrates what was predicted by the theory. Firstly, it shows that performing multiple rounds of communication improves the neighborhood of convergence. Secondly, it shows that there is a \textit{diminishing returns} effect of the number of communication rounds on the performance. Namely, the neighborhood improves substantially when going from $1$ consensus step to $2$ consensus steps, however, going from $5$ consensus steps to $10$ consensus steps has a much smaller effect. It is interesting to investigate that the performance of the methods in terms of the cost. Given a fixed cost budget, (e.g., $10^4$) it appears that only one method (blue: $2$ consensus steps per gradient step) is competitive with and better than the baseline DGD method. Again, the reason for this is the marginal returns effect of performing many more consensus steps per iteration. We should, of course, mention that our observation about the performance of the methods in terms of the cost are highly dependent on the cost structure that we chose for these experiments ($c_c = c_g = 1$). \textcolor{black}{In Section \ref{sec:ndgd_num_res}}, we show results for different cost structures.

Before proceeding forward we make one final remark. We tested the performance of the four methods on quadratic problems with different characteristic and different graph sizes. While the absolute performance of the methods changed, the relative performance of the methods did not. More specifically, with the cost structure $c_c = c_g = 1$, similar behavior as that displayed in Figure \ref{multiple_cons_DGD} was observed.

\section{NEAR-DGD}
\label{sec:dgd_variants}

Motivated by the improved results of the DGD$^t$ method and the power of performing multiple consensus steps, we ask the question whether a first-order distributed method can achieve exact convergence to the optimal solution of problem \eqref{eq:cons_prob1} by simply performing multiple rounds of communication. The results from the previous section suggest that a simple modification as in DGD$^t$ is not sufficient. To construct new algorithms, we observe that each iteration of DGD [cf.\ Eq.\ \eqref{eq:dgd_vec}] consists of two operators,
\begin{itemize}
		\item Consensus Operator: $\mathcal{W}[\mathbf{x}] = \mathbf{Zx}$,
		\item Gradient Operator: $\mathcal{T}[\mathbf{x}] = \mathbf{x} - \alpha \nabla \textbf{f}(\textbf{x}). $
\end{itemize}
Using these operators the DGD method can be expressed as
\begin{align}\label{eq:DGDOp}
	\textbf{x}_{k+1}= (\mathcal{T}-I + \mathcal{W})[\textbf{x}_k] = \textbf{Z}\textbf{x}_k - \alpha \nabla \textbf{f}(\textbf{x}_k).
\end{align}
\ew{We can view the consensus and gradient steps as two separable operations. This enables a decomposition of the computation and communication operations and allows for flexible customization in view of our new cost framework. An alternative way to combine these two operators is by nesting them. Simply alternating between these two operations leads to} our first new method, which we call the \ndgd method -- \textbf{N}ested \textbf{E}xact \textbf{A}lternating \textbf{R}ecursion method. The $\tau^{th}$ iterate of the method can be expressed as,
\begin{align}\label{eq:NDGD}
	\textbf{x}_\tau = [\mathcal{W}[\mathcal{T}[\cdots
	\rlap{$\overbrace{\phantom{[\mathcal{W}[\mathcal{T}[}}^{x_k}$}
	[\mathcal{W}
	\underbrace{[\mathcal{T}[\mathcal{W}[}_{y_k} \mathcal{T}[\cdots[\mathcal{W}[\mathcal{T}[\textbf{x}_0]]]\cdots]]]]]\cdots]]].
\end{align}
\ew{Each iteration of \ndgd involves the same amount of communication and computation as the standard DGD method. The main difference is that the gradient is now computed at the variable after the consensus step, i.e., the counterpart of Eq.\ \eqref{eq:DGDOp} is given by
\[	\textbf{x}_{k+1} = \textbf{Z}\textbf{x}_k - \alpha \nabla \textbf{f}(\textbf{Z}\textbf{x}_k).\]	Compared to the original DGD method \eqref{eq:dgd_vec}, where the gradient step and communication step can be done in parallel, newer information is used to compute the gradient step in NEAR-DGD, and thus it has two inherently sequential steps.
 Alternatively, we can view the \ndgd method as a method that} produces an intermediate iterate $\textbf{y}_k$ after the gradient step ($\mathcal{T}$), and the iterate $\textbf{x}_k$ after the consensus step ($\mathcal{W}$). The iterates $\textbf{x}_k$ and $\textbf{y}_k$ can be expressed as
\begin{gather*}	
	\textbf{x}_k  = \mathcal{W}[\mathbf{y}_k] = \textbf{Z}\textbf{y}_k \label{eq:tw} \\
	\textbf{y}_{k+1}  =\mathcal{T}[\mathbf{x}_k] =  \textbf{x}_k -   \alpha \nabla \textbf{f}(\textbf{x}_k). 
\end{gather*}
We assume here that the local iterates $x_{i,0}$ are initialized to be equal\footnote{With slightly more complex notation and algebra, we can show that similar results hold for either $[\mathcal{T}[\mathcal{W}[x]]$ or $[\mathcal{W}[\mathcal{T}[x]]$, in the case where the agents initialize at different points. }. As a result, we can start with either the $\mathcal{T}$ or $\mathcal{W}$ operation and have the same expression as Eq.\ \eqref{eq:NDGD}, since an initial consensus step would result in the same iterate ($\textbf{x}_0 = \textbf{Z}\textbf{y}_0 = \textbf{y}_0$).

As with the DGD method, we propose a variant of the \ndgd method that performs multiple consensus steps per gradient step. This method--which we call the NEAR-DGD$^t$-- can be expressed as
\begin{align*}
	\textbf{x}_\tau = [\mathcal{W}^t[\mathcal{T}[\cdots
	\rlap{$\overbrace{\phantom{[\mathcal{W}^t[\mathcal{T}[}}^{x_k}$}
	[\mathcal{W}^t
	\underbrace{[\mathcal{T}[\mathcal{W}^t[}_{y_k} \mathcal{T}[\cdots[\mathcal{W}^t[\mathcal{T}[\textbf{x}_0]]]\cdots]]]]]\cdots]]],
\end{align*}
where $\mathcal{W}^t[x]$ denotes $t$ nested consensus operations (steps),
\begin{align*}
	\mathcal{W}^t[x] = \underbrace{\mathcal{W}[\cdots[\mathcal{W}[\mathcal{W}}_{\text{t operations}}[x]]]\cdots].
\end{align*}
In terms of the iterates $\textbf{x}_k$ and $\textbf{y}_k$ the \ndgdt method can be expressed as
\begin{gather}	
	\textbf{x}_k  = \mathcal{W}^t[\mathbf{y}_k] =  \textbf{Z}^t\textbf{y}_k \label{eq:twt} \\
	\textbf{y}_{k+1}  =\mathcal{T}[\mathbf{x}_k] = \textbf{x}_k -   \alpha \nabla \textbf{f}(\textbf{x}_k). \label{eq:twt_step}
\end{gather}
The \ndgd method is a special case of the \ndgdt method\ew{, with $t=1$}. \ew{Given the flexibility in designing algorithms, we note that the number of consensus steps does not have to stay constant throughout the algorithm, hence we }also propose and analyze the \ndgdp method with time-varying consensus steps,
\begin{align*}
	\textbf{x}_\tau &= [\mathcal{W}^{t(\tau)}[\mathcal{T}[\cdots
	\rlap{$\overbrace{\phantom{[\mathcal{W}^{t(k)}[\mathcal{T}[}}^{x_k}$}
	[\mathcal{W}^{t(k)}
	\underbrace{[\mathcal{T}[\mathcal{W}^{t(k-1)}[}_{y_k} \mathcal{T}[\cdots\\
	&\quad\cdots[\mathcal{W}^{t(2)}[\mathcal{T}[\mathcal{W}^{t(1)}[\mathcal{T}[\textbf{x}_0]]]]]\cdots]]]]]\cdots]]],
\end{align*}
where $\mathcal{W}^{t(k)}[x]$ denotes $t(k)$ nested consensus operations (steps). In terms of the iterates $\textbf{x}_k$ and $\textbf{y}_k$ the \ndgdp method can be expressed as
\begin{gather}	
	\textbf{x}_k  = \mathcal{W}^{t(k)}[\mathbf{y}_k] =  \textbf{Z}^{t(k)}\textbf{y}_k \label{eq:tw+} \\
	\textbf{y}_{k+1}  =\mathcal{T}[\mathbf{x}_k]= \textbf{x}_k -   \alpha \nabla \textbf{f}(\textbf{x}_k). \label{eq:tw+_step}
\end{gather}
where at every iteration we change/increase the number of consensus steps ($t(k)$).

\subsection{Convergence Analysis of \ndgdt and \ndgdp}

We first analyze the \ndgdt method, 
 and then illustrate the convergence properties of the \ndgdp method.
We adopt the same assumptions (\ref{assm:Lip} \& \ref{assm:Strong}) as in Section \ref{sec:dgd}, and similarly to Section \ref{sec:dgd}, we define the average of $y_{i,k}$ as
\begin{align*}
	\bar{y}_k = \frac{1}{n}\sum_{i=1}^n y_{i,k}.
\end{align*}

We note that the gradient step in the \ndgd method, and by extension the \ndgdt and \ndgdp methods, can be viewed as a single step gradient iteration at the point $x_k$ on the following unconstrained problem
\begin{align}	\label{eq:tw_penalty}
\min_{x_i \in \mathbb{R}^p} \quad\sum_{i=1}^n f_i(x_i).
\end{align}
We use this observation to bound the iterates $\textbf{x}_k$ and $\textbf{y}_k$.

\begin{lem} 		\label{lem:twt_bound_iterates}
\textbf{(Bounded iterates)} Suppose Assumptions \ref{assm:Lip} \& \ref{assm:Strong} hold, and let the steplength satisfy
\begin{align*}
\alpha < \frac{1}{L}
\end{align*}
where  $L = \max_{i}L_i$.  Then, starting from $x_{i,0}= s_0$ or $y_{i,0}= s_0$ ($1 \leq i \leq n$), the iterates generated by the  \ndgdt method \eqref{eq:twt}-\eqref{eq:twt_step} are bounded, namely,
\begin{align*}
	\| \textbf{x}_k \| \leq D, \quad {\| \textbf{y}_k \| \leq D}
\end{align*}
for all $k=1,2,\ldots$, where $D = \| \textbf{y}_0 - \textbf{u}^\star\| + \frac{\nu + {4}}{\nu}\|\textbf{u}^\star \|$, $\textbf{u}^\star = [u_1^\star;u_2^\star;...;u_n^\star] \in \mathbb{R}^{np}$, $u_i^\star = \arg\min_{u_i}f_i(u_i)$, $\textbf{u}^\star$ is the optimal solution of \eqref{eq:tw_penalty}, $\nu = 2\alpha \textcolor{black}{\gamma}$, $\textcolor{black}{\gamma} = \min_i \textcolor{black}{\gamma}_i$ and \textcolor{black}{$\textcolor{black}{\gamma}_i = \frac{\mu_i L_i}{\mu_i + L_i}$}. 
\end{lem}

\begin{proof} 

Using standard results for the gradient descent method \cite[Theorem 2.1.15, Chapter 2]{nesterov2013introductory}, and noting that $\alpha < \frac{1}{L} < \frac{2}{\mu_i + L_i}$, which is the necessary condition on the steplength, we have for any $i \in \{1,2,...,n\}$
\begin{align*}
	\| x_{i,k} - \alpha \nabla {f}_i({x}_{i,k}) - {u_i}^\star\| &\leq \sqrt{1 - 2\alpha \textcolor{black}{\gamma}_i} \|x_{i,k} - {u_i}^\star\|.
\end{align*}
From this, we have,
\begin{align}
\| \textbf{x}_k - \alpha \nabla \textbf{f}(\textbf{x}_k) - \textbf{u}^\star\| &= \sqrt{\sum_{i=1}^n \| x_{i,k} - \alpha \nabla {f}_i({x}_{i,k}) - {u_i}^\star\|^2}\nonumber\\
	& \leq \sqrt{\sum_{i=1}^n (1 - 2\alpha \textcolor{black}{\gamma}_i) \|x_{i,k} - {u_i}^\star\|^2}\nonumber\\
	& \leq \sqrt{ (1- \nu)} \| \textbf{x}_k - \textbf{u}^\star \| \label{eq:dgdt_gradd}.
\end{align}
where the last inequality follows from the definition of $\nu$.


Using the definitions of $\nu$, $\textbf{y}_{k+1}$ and Eq. \eqref{eq:dgdt_gradd}, we have
\begin{align*}
	\| \textbf{y}_{k+1} - \textbf{u}^\star \| &= \| \textbf{x}_k - \alpha \nabla \textbf{f}(\textbf{x}_k) - \textbf{u}^\star \|\\
	& \leq \sqrt{(1-\nu)} \| \textbf{x}_k -\textbf{u}^\star\|\\
	& = \sqrt{(1-\nu)} \| \textbf{Z}^t\textbf{y}_k -\textbf{u}^\star\|\\
		& \leq \sqrt{(1-\nu)} \left[\|\textbf{Z}^t\| \| \textbf{y}_k -\textbf{u}^\star\| + \|I-\textbf{Z}^t\| \| \textbf{u}^\star\| \right].
\end{align*}
The eigenvalues of matrix $\textbf{Z}^t$ are the same as those of matrix $\textbf{W}^t$. The spectrum property of $\textbf W$ guarantees that the magnitude of each eigenvalue is upper bounded by 1. Hence $\|\textbf{Z}\|\leq 1$ and $\|I-\textbf{Z}^t\|\leq 2$ for all $t$. Hence the above relation implies that
\[	\| \textbf{y}_{k+1} - \textbf{u}^\star \| \leq \sqrt{(1-\nu)}  \| \textbf{y}_k -\textbf{u}^\star\| +  2\sqrt{(1-\nu)}\| \textbf{u}^\star\|. \]
Recursive application of the above relation gives,
\begin{align*}
	\| &\textbf{y}_{k+1} - \textbf{u}^\star \| \\&\leq (1-\nu)^{(k+1)/2}  \| \textbf{y}_0 -\textbf{u}^\star\| +  2\sum_{j=0}^k(1-\nu)^{(j+1)/2}\| \textbf{u}^\star\|\\
	& \leq   \| \textbf{y}_0 -  \textbf{u}^\star\| +  \frac{2\sqrt{1 - \nu}}{1 - \sqrt{1 -\nu}}\| \textbf{u}^\star\| \\
	& \leq   \| \textbf{y}_0 -  \textbf{u}^\star\| +  \frac{4}{\nu}\| \textbf{u}^\star\|.
\end{align*}
Thus, we bound the iterate as
\begin{align*}
	\| \textbf{y}_{k+1} \| &\leq	\| \textbf{y}_{k+1} - \textbf{u}^\star \| + \| \textbf{u}^\star\|\\
	& \leq  \| \textbf{y}_0 - \textbf{u}^\star\| + \frac{\nu + 4}{\nu}\|\textbf{u}^\star \|.
\end{align*}

We now show that the same result is true for the iterates $\textbf{x}_k$. Using the definition of $\textbf{x}_k$ Eq. \eqref{eq:twt}
\begin{align*}
	\| \textbf{x}_{k+1}\| &= \| \textbf{Z}^t \textbf{y}_{k+1} \|\\
	& \leq \| \textbf{Z}^t\| \| \textbf{y}_{k+1} \|\\
	& \leq \| \textbf{y}_{k+1} \|\\
	& \leq D,
\end{align*}
which completes the proof.
\end{proof}

Lemma \ref{lem:twt_bound_iterates} shows that the iterates generated by the \ndgdt method are bounded. Since eigenvalues of $Z^t$ and $I-Z^t$ are bounded above by 1 and 2, for any $t$, respectively, the same analysis can be used to show that the iterates generated by the \ndgdp method are also bounded.

\begin{lem} 	\label{lem:twt_bound_dev_mean}
\textbf{(Bounded deviation from mean)} If Assumptions \ref{assm:Lip} \& \ref{assm:Strong} hold. Then, starting from $x_{i,0}= s_0$ or $y_{i,0}= s_0$ ($1 \leq i \leq n$), the total deviation of each agent's estimate ($x_{i,k}$) from the mean is bounded, namely,
\begin{gather}		\label{eq:twt_lem2_p1}
	{\| x_{i,k} - \bar{x}_k\| \leq \beta^{{t}} D }
\end{gather}
and
\begin{gather}
	{\| \nabla f_i(x_{i,k}) - \nabla f_i(\bar{x}_k)\| \leq \beta^{{t}} D L_i} \label{eq:twt_lem2_p2}\\
	{\| g_k - \bar{g}_k \| \leq\beta^{{t}} D L}	\label{eq:twt_lem2_p3}
\end{gather}
for all $k=1,2,\ldots$ and $1 \leq i \leq n$. Moreover, the total deviation of the local iterates $y_{i,k}$ is also bounded,
\begin{align}	\label{eq:twt_lem2_p4}
	\| y_{i,k} - \bar{y}_k\| \leq \beta^{{t}} D  + 2D.
\end{align}
\end{lem}

\begin{proof}
Consider,
\begin{align*}
	\|x_{i,k} - \bar{x}_k\| & = \|x_{i,k} - \bar{y}_k\| \\
	&\leq \left\|\textbf{x}_k - \frac{1}{n}\left((1_n1_n^T)\otimes I \right) \textbf{y}_k\right\| \\
	&=\left\|(\textbf{W}^{t} \otimes I)\textbf{y}_k  - \frac{1}{n}\left((1_n1_n^T)\otimes I\right) \textbf{y}_k \right\| \\
	&\leq \left\|\left(\textbf{W}^{t}  - \frac{1}{n}\left((1_n1_n^T) \right)\otimes I \right)\right\|\|\textbf{y}_k\| \\
	&\leq \beta^t \|\textbf{y}_k\| \leq \beta^{{t}} D,
\end{align*}
where the first equality is due to the fact that $\bar{x}_k = \textbf{Z}^t\bar{y}_k = \bar{y}_k$ and the last inequality is due to Lemma \ref{lem:twt_bound_iterates}.

The result \eqref{eq:twt_lem2_p2} is a direct consequence of the \eqref{eq:twt_lem2_p1} and the Lipschitz continuity of individual gradients (Assumption \ref{assm:Lip}). To establish the next result \eqref{eq:twt_lem2_p3}, we have
\begin{align*}
\|g_k - \bar{g}_k\| &= \left\|\frac{1}{n} \sum_{i=1}^{n} \left(\nabla f_i(x_{i,k}) -  \nabla f_i(\bar{x}_k)\right) \right\| \\
& \leq \frac{1}{n} \sum_{i=1}^{n} L_i \|x_{i,k} - \bar{x}_k\| \leq \beta^{{t}} D L.
\end{align*}

Finally, for the local $y_{i,k}$ iterates in  \eqref{eq:twt_lem2_p4}, consider
\begin{align*}
	\|y_{i,k} - \bar{y}_k\| & \leq \| x_{i,k} - \bar{y}_k\| + \| y_{i,k} - x_{i,k}\|\\
	& \leq \beta^{{t}} D + \| \textbf{y}_k - \textbf{x}_k \| \\
	&=  \beta^{{t}} D + \left\| \textbf{y}_k - \left( \textbf{W}^t \otimes I\right)\textbf{y}_k \right\| \\
	&\leq  \beta^{{t}} D + \left\| \left(I - \textbf{W}^t \otimes I\right) \right\| \|\textbf{y}_k \| \\
	&\leq  \beta^{{t}} D + 2D,
\end{align*}
where the second inequality is due to \eqref{eq:twt_lem2_p1} and the last inequality is due to Lemma \ref{lem:twt_bound_iterates}.
\end{proof}

Similar to Lemma \ref{lem:dgdt_bound_dev_mean}, Lemma \ref{lem:twt_bound_dev_mean} shows that the distance between the local iterates $x_{i,k}$ and $y_{i,k}$ are bounded from their means.

We now use an argument similar to that in the previous section to investigate the optimization error of the \ndgdt method. For that we make the following observation, due to the doubly-stochastic nature of $\textbf{W}$,
\begin{align*}
	\bar{\textbf{y}}_{k+1} &= \frac{1}{n}\left((1_n1_n^T) \otimes I \right){\textbf{y}}_{k+1} \\
	&= \frac{1}{n}\left((1_n1_n^T) \otimes I \right) \left((\textbf{W}^{{t}} \otimes I )\textbf{y}_{k} - \alpha \nabla \textbf{f}(\textbf{x}_k) \right)\\
	& = \frac{1}{n}\left((1_n1_n^T\textbf{W}^{{t}}) \otimes I \right)\textbf{y}_{k} -  \frac{\alpha}{n} \left((1_n1_n^T) \otimes I \right) \nabla \textbf{f}(\textbf{x}_k)\\
	& = [\bar{y}_k - \alpha g_k; \bar{y}_k - \alpha g_k;\ldots;\bar{y}_k - \alpha g_k],
\end{align*}
where $g_k$ is the average of local gradients as defined in \eqref{eq:g_barg}.
Thus we have
\begin{align}	\label{eq:twt_errors}
	\bar{y}_{k+1} = \bar{y}_k - \alpha g_k.
\end{align}
The above equation can be viewed as an inexact gradient descent step for the problem \eqref{eq:prob_bar}, where $\bar{g}_k$ is the exact gradient. Before we proceed we should note again that $\bar{x}_k = \textbf{Z}^t\bar{y}_k = \bar{y}_k$ due to the nature of the matrix $\textbf{Z}$. We next follow Theorem \ref{thm:dgdt_bound_dist_min} to bound the distance of iterates to the optimal solution.

\begin{thm} \label{thm:twt_bound_dist_min}
\textbf{(Bounded distance to minimum)} Suppose Assumptions \ref{assm:Lip} \& \ref{assm:Strong} hold, and let the steplength satisfy
\begin{align*}
		\alpha \leq \min \left \{ {\frac{1}{L}, c_4} \right \}
\end{align*}
where  $L = \max_{i}L_i$ and $c_4 = \frac{\textcolor{black}{2}}{\mu_{\bar{f}} +L_{\bar{f}}}$. Then, starting from $x_{i,0} = s_0$ or $y_{i,0} = s_0$ ($1\leq i \leq n$), for all $k=0,1,2,\ldots$
\begin{align*}
	\| \bar{x}_{k+1} - x^\star\|^2 \leq c_1^2 \| \bar{x}_{k} - x^\star \|^2 + {c_3^2\beta^{2{t}}},
\end{align*}
where
\begin{align*}
	c_1^2 = 1 - \alpha c_2 + \alpha \delta - \alpha^2 \delta c_2, \;\; c_2 = \frac{\textcolor{black}{2}\mu_{\bar{f}} L_{\bar{f}}}{\mu_{\bar{f}} + L_{\bar{f}}},\\
	c_3^2 = \alpha(\alpha + \delta^{-1}){ D^2 L^2}, \;\; D = \| \textbf{y}_0 - \textbf{u}^\star\| + \frac{\nu + {4}}{\nu}\|\textbf{u}^\star \|,
\end{align*}
$x^\star$ is the optimal solution of \eqref{eq:cons_prob1}, $\textbf{u}^\star$ is the optimal solution of \eqref{eq:tw_penalty}, $\nu = 2\alpha \textcolor{black}{\gamma}$, $\textcolor{black}{\gamma} = \min_i \textcolor{black}{\gamma}_i$, {$\textcolor{black}{\gamma}_i = \frac{\mu_i L_i}{\mu_i + L_i}$} and $\delta>0$. In particular, if we set $\delta = \frac{c_2}{2(1-\alpha c_2)}$ such that $c_1 = \sqrt{1 - \frac{\alpha c_2}{2}} \in (0,1)$, then for $k=0,1,2,\ldots$
 \begin{align*}
 	\| \bar{x}_{k} - x^\star\| \leq c_1^k \| \bar{x}_{0} - x^\star\| + \textcolor{black}{\frac{LD\beta^t}{c_2} \sqrt{2(2-\alpha c_2)}}.
\end{align*}
\end{thm}

\begin{proof} Following the analysis of Theorem \ref{thm:dgdt_bound_dist_min}, and
using the definitions of the $\bar{x}_k$ and $g_k$, and \eqref{eq:twt_errors}, we have
\begin{align}		
	\| \bar{x}_{k+1} - x^\star \|^2 &= \| \bar{x}_{k} - x^\star - \alpha g_k\|^2 \nonumber\\
		& \leq (1 + \alpha \delta)\| \bar{x}_{k} - x^\star - \alpha \bar{g}_k\|^2 \nonumber\\
		& \quad+ \alpha(\alpha + \delta^{-1})\| \bar{g}_k - {g}_k \|^2, \label{eq:twt_thm1}
\end{align}
where the last inequality follows from the fact that for any vectors $a$ and $b$, $\pm 2a^Tb \leq \delta^{-1} \| a\|^2 + \delta \| b \|^2$, for $\delta>0$.

We now bound the quantity $\| \bar{x}_{k} - x^\star - \alpha \bar{g}_k\|^2$ as
\begin{align}		\label{eq:neardgdt_thm_bound_distance}
	\| &\bar{x}_{k} - x^\star - \alpha \bar{g}_k\|^2 \nonumber \\ &= \|  \bar{x}_{k} - x^\star \|^2 + \alpha^2 \|  \bar{g}_k\|^2 - 2(\bar{x}_{k} - x^\star)^T(\alpha \bar{g}_k) \nonumber \\
	& \leq \|  \bar{x}_{k} - x^\star \|^2 + \alpha^2 \| \bar{g}_k\|^2 - \alpha c_4 \| \bar{g}_k\|^2 - \alpha c_2 \| \bar{x}_{k} - x^\star \|^2 \nonumber \\
	& = (1 - \alpha c_2)\|  \bar{x}_{k} - x^\star \|^2 + \alpha(\alpha -  c_4)\| \bar{g}_k\|^2 \nonumber\\
	& \leq (1 - \alpha c_2)\|  \bar{x}_{k} - x^\star \|^2,
\end{align}
{where the first inequality follows from \cite[Theorem 2.1.12, Chapter 2]{nesterov2013introductory}, and in the last inequality we dropped the term $\alpha(\alpha -  c_4)\| \bar{g}_k\|^2$, since $\alpha\leq c_4$ and $\alpha(\alpha -  c_4)\| \bar{g}_k\|^2 \leq 0$.}
Combining \eqref{eq:twt_thm1}, \eqref{eq:neardgdt_thm_bound_distance} and using \eqref{eq:twt_lem2_p3},
\begin{align}			\label{eq:twt_onestep}
	\| \bar{x}_{k+1} - x^\star \|^2 & \leq (1 + \alpha \delta)(1 - \alpha c_2)\|  \bar{x}_{k} - x^\star \|^2 \nonumber\\
	& \quad + \alpha(\alpha + \delta^{-1})  L^2 D^2 {\beta^{2{t}}}.
\end{align}

Using the definitions of $c_1$ and $c_3$, by recursive application of \eqref{eq:twt_onestep}, we have
\begin{align*}
	\| \bar{x}_{k} - x^\star \|^2 &\leq c_1^{2k} \| \bar{x}_{0} - x^\star\|^2 + \frac{c_3^2}{(1-c_1^2)}\beta ^{2t},
\end{align*}
and so
\begin{align*}
	\| \bar{x}_{k} - x^\star \| &\leq c_1^{k} \| \bar{x}_{0} - x^\star\| + \frac{c_3}{\sqrt{1-c_1^2}}\beta^t.
\end{align*}
If $\delta = \frac{c_2}{2(1-\alpha c_2)}$, then
\begin{align*}
	c_1^2 = 1 - \frac{\alpha c_2}{2}<1,
\end{align*}
and
\begin{align*}
	\frac{c_3}{\sqrt{1-c_1^2}}\beta^t = \textcolor{black}{\frac{LD\beta^t}{c_2} \sqrt{2(2-\alpha c_2)}},
\end{align*}
which completes the proof.
\end{proof}

Theorem \ref{thm:twt_bound_dist_min} show that the average of the iterates generated by the \ndgdt method converge to a neighborhood of the optimal solution whose size is defined by the steplength, the second largest eigenvalue of $\textbf{W}$ and the number of consensus steps. We now provide a convergence result for the local agent estimates of the \ndgdt method.

\begin{cor}		\label{cor:twt_bound_dist_min}
\textbf{(Local agent convergence)} Suppose Assumptions \ref{assm:Lip} \& \ref{assm:Strong} hold, and let the steplength satisfy
\begin{align*}
		\alpha \leq \min \left \{ {\frac{1}{L}, c_4} \right \}.
\end{align*}
where $L = \max_{i}L_i$ and $c_4 = \frac{{2}}{\mu_{\bar{f}} +L_{\bar{f}}}$. Then, starting from $x_{i,0} = s_0$ or $y_{i,0} = s_0$ ($1\leq i \leq n$) for $k=0,1,\ldots$
\begin{align*}
	\| x_{i,k} - x^\star \| \leq c_1^k \| x_0 -  x^\star\| + \frac{c_3}{\sqrt{1-c_1^2}}\beta^t + \beta^{{t}} D.
\end{align*}
Moreover, the local iterates $y_{i,k}$ are bounded by
\begin{align*}
	\| y_{i,k} - x^\star \| \leq c_1^k \| x_0 -  x^\star\| + \frac{c_3}{\sqrt{1-c_1^2}}\beta^t + \beta^{{t}} D + 2D.
\end{align*}
\end{cor}

\begin{proof} Using the results from Lemma \ref{lem:twt_bound_dev_mean} and Theorem \ref{thm:twt_bound_dist_min},
\begin{align*}
	\| x_{i,k} - x^\star \| &\leq \| \bar{x}_k - x^\star \| + \| x_{i,k} - \bar{x}_k \| \\
		& \leq c_1^k \| x_0 - x^\star\| + \frac{c_3}{\sqrt{1-c_1^2}}\beta^t + \beta^{{t}} D.
\end{align*}

Following the same approach for the local iterates $y_{i,k}$, we have
\begin{align*}
	\| y_{i,k} - x^\star \| &\leq \| \bar{y}_k - x^\star \| + \| y_{i,k} - \bar{y}_k \| \\
	&= \| \bar{x}_k - x^\star \| + \| y_{i,k} - \bar{y}_k \| \\
		& \leq c_1^k \| x_0 - x^\star\| + \frac{c_3}{\sqrt{1-c_1^2}}\beta^t + \beta^{{t}} D + 2D,
\end{align*}
where the equality is due to $\bar{x}_k = \textbf{Z}^t\bar{y}_k = \bar{y}_k$.
\end{proof}

The main takeaway of Theorem \ref{thm:twt_bound_dist_min} is that the iterates generated by the \ndgdt method converge at a linear rate to a neighborhood of the optimal solution, where the neighborhood is defined as,
\begin{align*}
	c_3^2\beta^{2{t}}.
\end{align*}
A natural question to ask is whether there is a way to increase the number of consensus steps at every iteration in order to eliminate the error term. In the next theorem, we show that this can actually be achieved, and that the \ndgdp method with an appropriate increase in the number of consensus steps converges at an R-Linear rate  to the solution. Before we proceed, we should mention that the results of Lemmas \ref{lem:twt_bound_iterates} and \ref{lem:twt_bound_dev_mean} extend naturally to the case with increasing number of consensus steps.

\begin{thm} \label{thm:tw+_bound_dist_min}
\textbf{(Bounded distance to minimum)} Suppose Assumptions \ref{assm:Lip} \& \ref{assm:Strong} hold, and let the steplength satisfy
\begin{align*}
		\alpha \leq \min \left \{ \frac{1}{L}, c_4 \right \}
\end{align*}
where $L = \max_{i}L_i$ and $c_4 = \frac{{2}}{\mu_{\bar{f}} +L_{\bar{f}}}$. Then, starting from $x_{i,0} = s_0$ or $y_{i,0} = s_0$ ($1\leq i \leq n$), for all $k=0,1,2,\ldots$
\begin{align*}
	\| \bar{x}_{k+1} - x^\star\|^2 \leq c_1^2 \| \bar{x}_{k} - x^\star \|^2 + {c_3^2\beta^{2{t(k)}}},
\end{align*}
where
\begin{align*}
	c_1^2 = 1 - \alpha c_2 + \alpha \delta - \alpha^2 \delta c_2, \;\; c_2 = \frac{{2}\mu_{\bar{f}} L_{\bar{f}}}{\mu_{\bar{f}} + L_{\bar{f}}},\\
	c_3^2 = \alpha(\alpha + \delta^{-1}){ D^2 L^2}, \;\; D = \| \textbf{y}_0 - \textbf{u}^\star\| + \frac{\nu + {4}}{\nu}\|\textbf{u}^\star \|,
\end{align*}
$x^\star$ is the optimal solution of \eqref{eq:cons_prob1}, $\textbf{u}^\star$ is the optimal solution of \eqref{eq:tw_penalty}, $\nu = 2\alpha \textcolor{black}{\gamma}$, $\textcolor{black}{\gamma} = \min_i \textcolor{black}{\gamma}_i$, {$\textcolor{black}{\gamma}_i = \frac{\mu_i L_i}{\mu_i + L_i}$} and $\delta>0$. \ew{Moreover, for any strictly increasing sequence $\{t(k)\}_k$, with $\lim_{k \rightarrow \infty} t(k) \rightarrow \infty$, the iterates produced by the \ndgdp algorithm converge to $x^\star$.  }
\end{thm}

\begin{proof} The proof of Theorem \ref{thm:tw+_bound_dist_min} is exactly the same as that of Theorem \ref{thm:twt_bound_dist_min}, with the difference that the constant number of consensus steps $t$ is replaced by a varying number of consensus steps $t(k)$. \ew{The convergence result follows from the fact that
	\[\lim_{k\to\infty}\beta^{2t(k)}=0,\] for any increasing sequence $\{t(k)\}$ with $\lim_{k \rightarrow \infty} t(k) \rightarrow \infty$, and thus the size of the error neighborhood $\mathcal{O}( \beta^{2t(k)})$ shrinks to 0.}
\end{proof}

\begin{thm} \label{thm:tw+_rlinear}
\textbf{(R-Linear convergence of the \ndgdp method)} Suppose Assumptions \ref{assm:Lip} \& \ref{assm:Strong} hold, let the steplength satisfy
\begin{align*}
		\alpha \leq \min \left \{ \frac{1}{L}, c_4 \right \}
\end{align*}
where $L = \max_{i}L_i$ and $c_4 = \frac{{2}}{\mu_{\bar{f}} +L_{\bar{f}}}$, and let ${{t(k) =k}}$. Then, starting from $x_{i,0} = s_0$ or $y_{i,0} = s_0$ ($1\leq i \leq n$), the iterates generated by the \ndgdp method \eqref{eq:tw+}-\eqref{eq:tw+_step} converge at an R-Linear rate to the solution. Namely,
\begin{align}
	\| \bar{x}_k - x^\star \| \leq C \rho^k
\end{align}
for all $k=0,1,2,...$, where
\begin{align*}
C = \max \left\{ \| \bar{x}_0 -x^\star \|, \frac{2c_3}{ \sqrt{\alpha c_2}}\right\}, \quad \rho = \max\left\{ \beta, \sqrt{1-\frac{\alpha c_2}{4}}\right\},
\end{align*}
and $c_1$, $c_2$, $c_3$ and $c_4$ are given in Theorem \ref{thm:tw+_bound_dist_min}.
\end{thm}

\begin{proof} We prove the result by induction. By the definitions of $C$ and $\rho$ the base case $k=0$ holds. Assume that the result is true for the $k^{th}$ iteration, and consider the $(k+1)^{th}$ iteration. Using the result of Theorem \ref{thm:tw+_bound_dist_min}, we have
\begin{align*}
	\| \bar{x}_{k+1} - x^\star \|^2 &\leq c_1^2 \| \bar{x}_{k} - x^\star \|^2 + c_3^2\beta^{2{k}}\\
	& \leq c_1^2 \left(C\rho^{k}\right)^2 + c_3^2\beta^{2k}\\
	& =  \left(C\rho^{k}\right)^2 \left[ c_1^2 + \frac{c_3^2 \beta^{2k}}{(C\rho^{k})^2}\right]\\
	& \leq \left(C\rho^{k}\right)^2 \left[c_1^2 + \frac{c_3^2}{C^2}\right]\\
	& \leq \left(C\rho^{k}\right)^2\left[ 1- \frac{\alpha c_2}{2} + \frac{\alpha c_2}{4}\right]\\
	& \leq \left(C\rho^{k+1} \right)^2
\end{align*}
where the third inequality is due to the fact that $\rho\geq \beta$, the fourth inequality is due to the definitions of $c_1$, relations $C\geq \frac{2c_3}{ \sqrt{\alpha c_2}}$ and $\alpha \delta - \alpha^2 \delta c_2\leq 0$, and the last inequality is due to the definition of $\rho$, where $\rho \geq \sqrt{1-\frac{\alpha c_2}{4}}$.
\end{proof}

Theorem \ref{thm:tw+_rlinear} illustrates that when the number of consensus steps is increased at an appropriate rate ($t(k)=k$) then the \ndgdp method converges to the solution at an R-Linear rate. Going back to Corollary \ref{cor:twt_bound_dist_min}, the result implies that the local iterates $x_{i,k}$ generated by \ndgdp method converge to the optimal solution, whereas the local iterates $y_{i,k}$ do not.

We now investigate the work complexity of the \ndgdp method. By work complexity we mean the total amount of work (gradient evaluations and communication steps) required to get an $\epsilon$-accurate solution ($\| \bar{x}_k - x^\star\| \leq \epsilon$).

\begin{cor} \textbf{(Work Complexity)} If the conditions in Theorem \ref{thm:tw+_rlinear} are satisfied, then the work complexity (total number of gradient evaluations $\tau_g$ and rounds of communications $\tau_c$) to get an $\epsilon$-accurate solution, that is $\| \bar{x}_k - x^\star\| \leq \epsilon$, for the \ndgdp algorithm are given as follows,
\begin{align*}
	\tau_g = \mathcal{O}\left(\log \left(\sfrac{1}{\epsilon}\right)\right),\\
	\tau_c = \mathcal{O}\left(\left(\log \left(\sfrac{1}{\epsilon}\right)\right)^2\right).
\end{align*}
\end{cor}
\begin{proof}
	Due to the result of Theorem \ref{thm:tw+_rlinear}, we require $\mathcal{O}\left(\log (\sfrac{1}{\epsilon})\right)$ iterations to get an $\epsilon$-accurate solution, and so the number of gradient evalutions ($\tau_g$) is $\mathcal{O}\left(\log (\sfrac{1}{\epsilon})\right)$. Since we require $k$ communications at the $k^{th}$ iterate, the total number of communications ($\tau_c$) required
	\begin{align*}
	\tau_c = \sum_{i=1}^k i = \frac{(k)(k+1)}{2} = \mathcal{O}\left(k^2\right) = \mathcal{O}\left(\left(\log (\sfrac{1}{\epsilon})\right)^2\right).
	\end{align*}
\end{proof}
Similar analysis can be done to show the work complexity required to get an $\epsilon$-accurate solution for the local iterates. Note, that this can only be done for the local iterates $x_{i,k}$, but not the local iterates $y_{i,k}$ as these iterates do not converge. These results can then be used in our cost framework \eqref{eq:cost} to obtain the final cost of the algorithm.

\subsection{Numerical Results for \ndgd and \ndgdp}
\label{sec:ndgd_num_res}

In this section, we present numerical results demonstrating the performance of the \ndgdt and  \ndgdp methods in practice on quadratic problems and classification problems that arise in machine learning. The aim of this section is to demonstrate that the theoretical results can be realized in practice. More specifically, that the \ndgdt method converges to a neighborhood of the solution and that the \ndgdp method converges to the solution, for objective functions that are strongly convex.

We investigated the performance of $6$ different variants of the \ndgd method and compared against the DGD method. We define the variants of the \ndgd method as \ndgd$(a,b,c)$, where $a$ is the number of initial gradient steps, $b$ is the number of initial consensus steps and $c$ describes if/how we increase the number of communication steps. For example, (i) \ndgd$(10,1,-)$ is the \ndgd method with $10$ gradient steps for every $1$ consensus step with fixed number of consensus and gradient steps; (ii) \ndgdp$(1,1,k)$ is the \ndgdp method with $1$ gradient step and $1$ consensus step initially, and where the number of consensus steps is increased at every iteration (at the $k^{th}$ iteration $k$ consensus steps are performed for every gradient step); and (iii) \ndgdp$(1,1,500)$ is the \ndgdp method with $1$ gradient step and $1$ consensus step initially, and the number of consensus steps is doubled every $500$ iterations.  The last class of methods are practical implementations of NEAR-DGD$^+$, which we found to perform well in the numerical studies. While the theoretical analysis in this paper does not apply to  \ndgd(10,1,-), a method with multiple gradient steps at each iteration, we include it in our numerical studies for comparison purposes. 

\subsubsection{Quadratic Problems}

We first investigate the performance of the methods on quadratic functions, similar to those described in Section \ref{sec:num_res_dgdt},
\begin{align*}
		 f(x) = \frac{1}{2} \sum_{i=1}^n x^TA_ix + b_i^Tx
\end{align*}
where each node $i=\{1,...,n\}$ has local information $A_i \in \mathbb{R}^{p\times p}$ and $b_i \in \mathbb{R}^{p} $. For these experiments we chose a dimension size $p=10$, the number of nodes was $n=10$ and the condition number was $\kappa = 10^4$. We considered a $4$-cyclic graph topology (each node is connected to its $4$ immediate neighbors). The markers in the figures in this section are placed every $500$ iterations. In this regard, one can clearly see the effect of the cost per iteration for the different methods. For example, in the \ndgdp$(1,1,k)$ method (dark green lines) of Figure \ref{fig:ndgd_quad}, in terms of iterations we have markers, whereas in terms of cost there are no markers. Of course, this is due to the high communication  cost associated with each iteration.

Figure \ref{fig:ndgd_quad} illustrates the performance of the methods; we again plot relative error, ($\|\bar{x}_k - x^\star\|^2/\| x^\star \|^2$) in terms of: (i) iterations, (ii) cost (as described in Section \ref{sec:cost}, with $c_c = c_g = 1$), (iii) number of gradient evaluations, and (iv) number of communications.

%

\begin{figure}[]
\centering

\includegraphics[width=0.41\columnwidth]{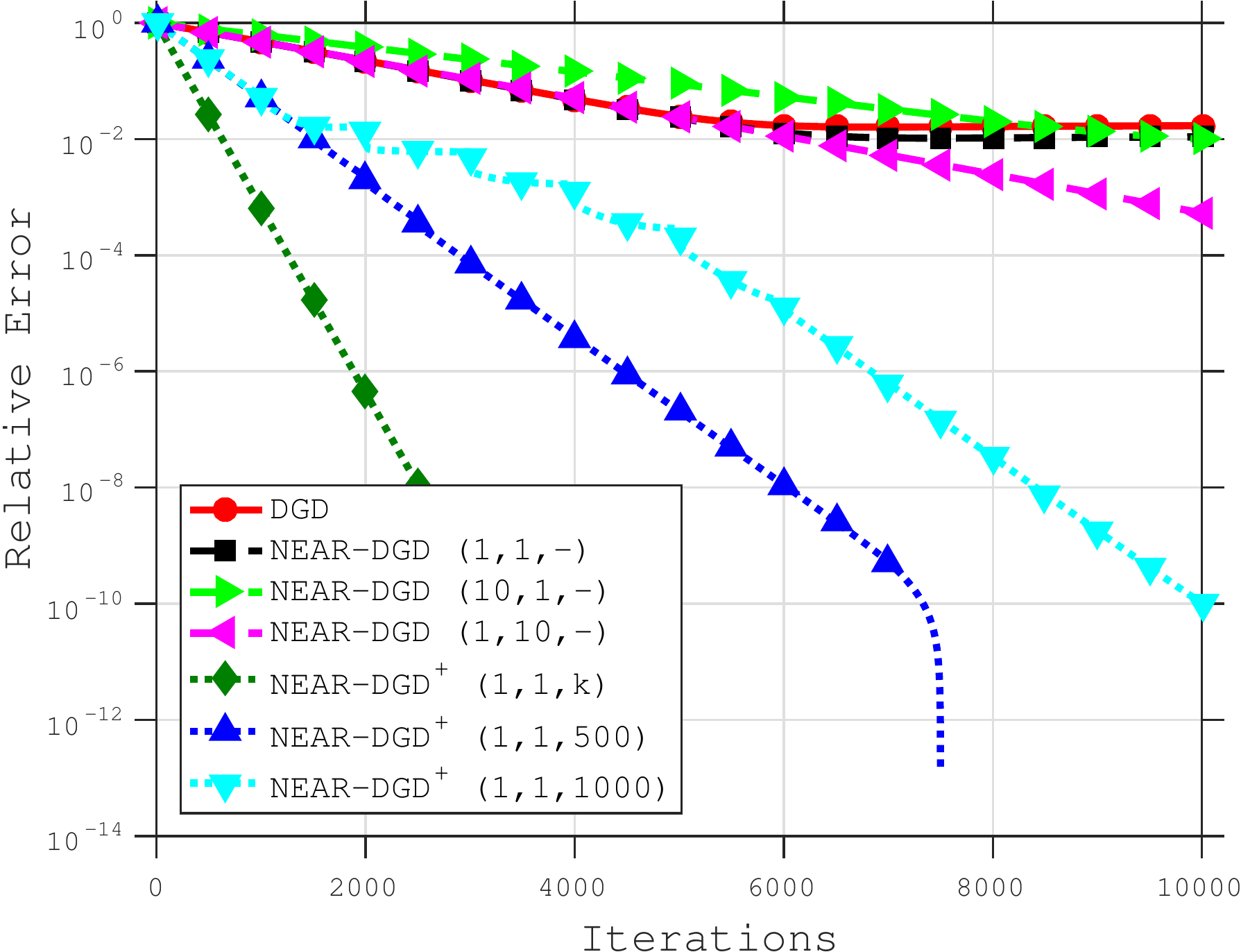}
\hspace{0.2cm}
\includegraphics[width=0.41\columnwidth]{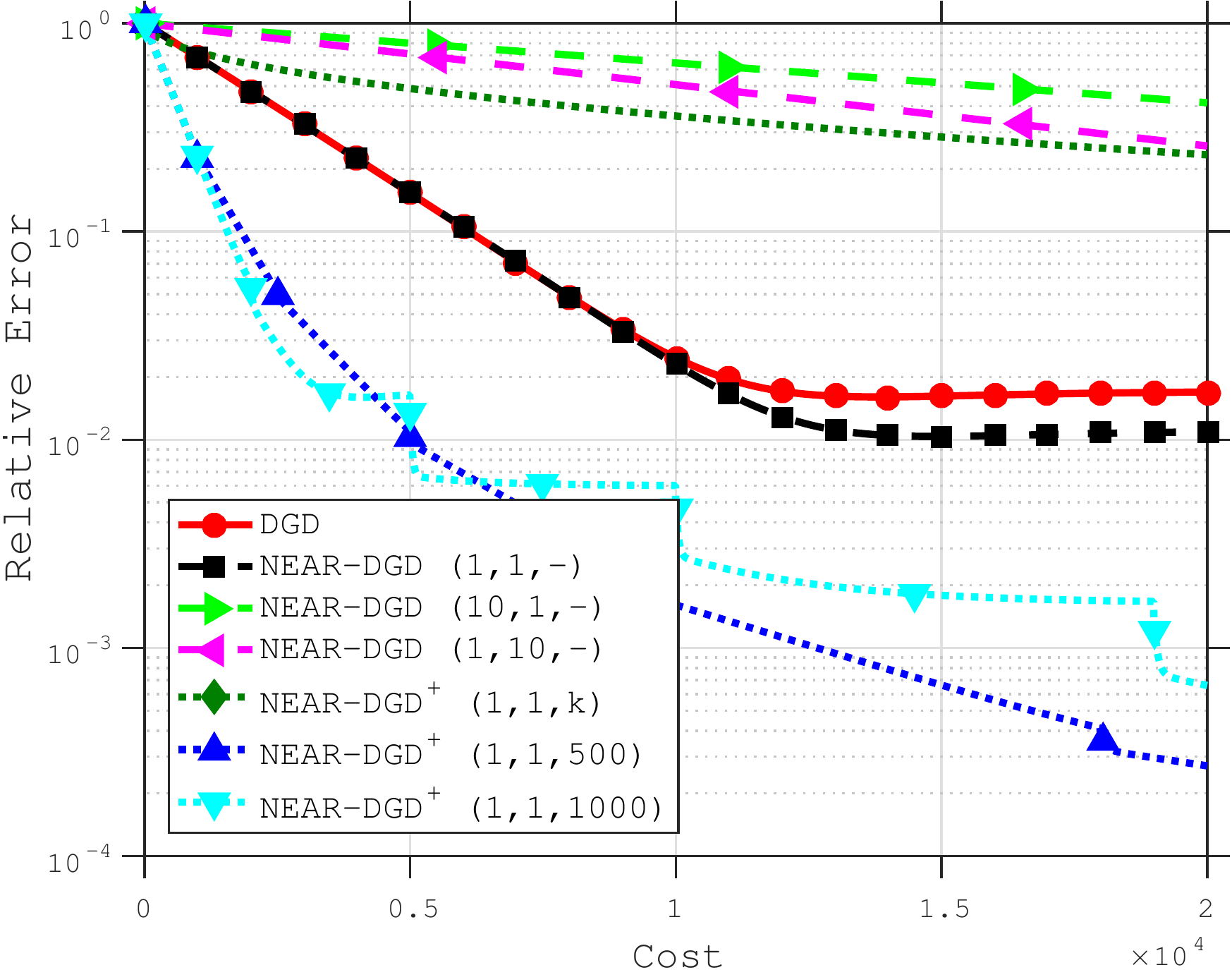}
\includegraphics[width=0.41\columnwidth]{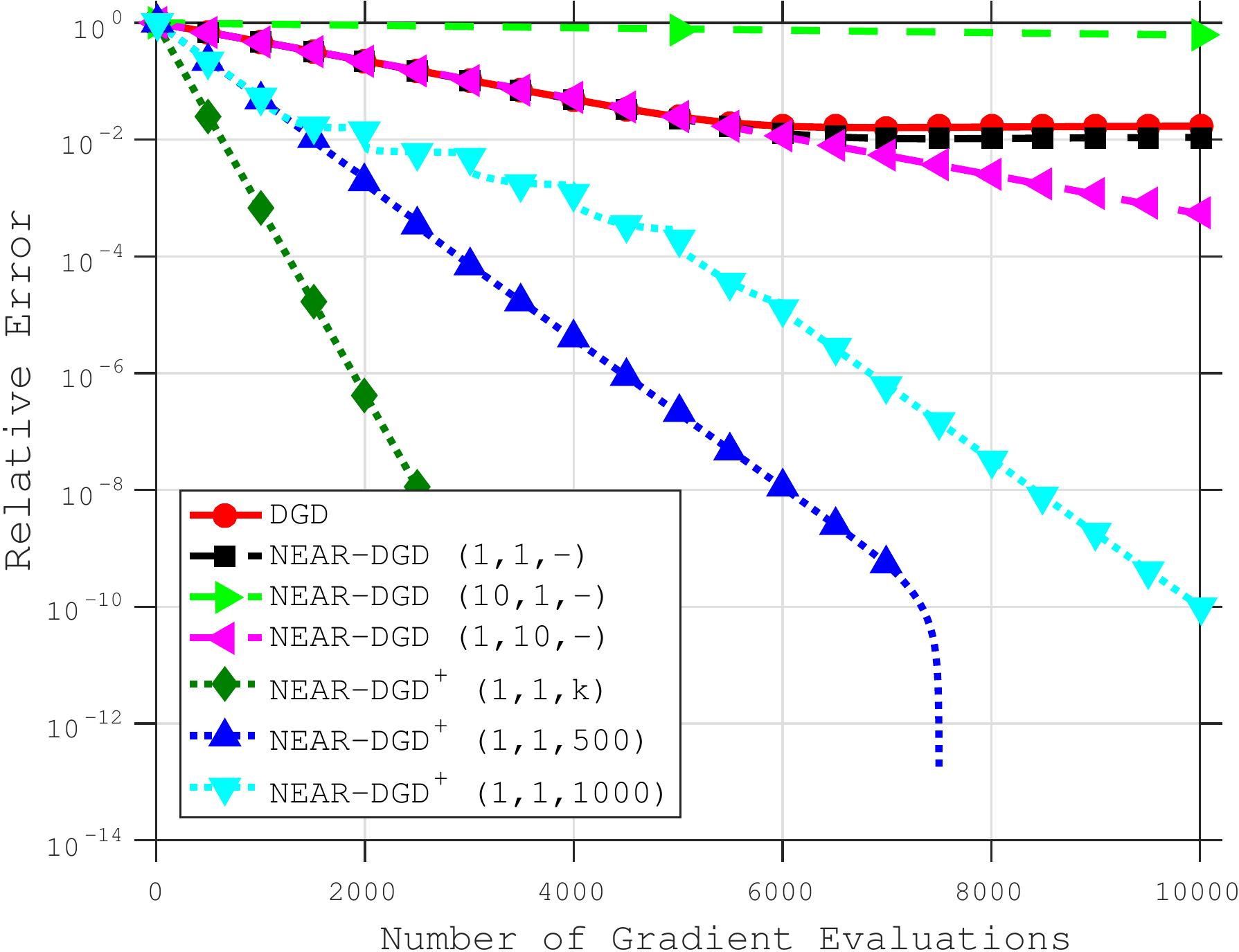}
\hspace{0.2cm}
\includegraphics[width=0.41\columnwidth]{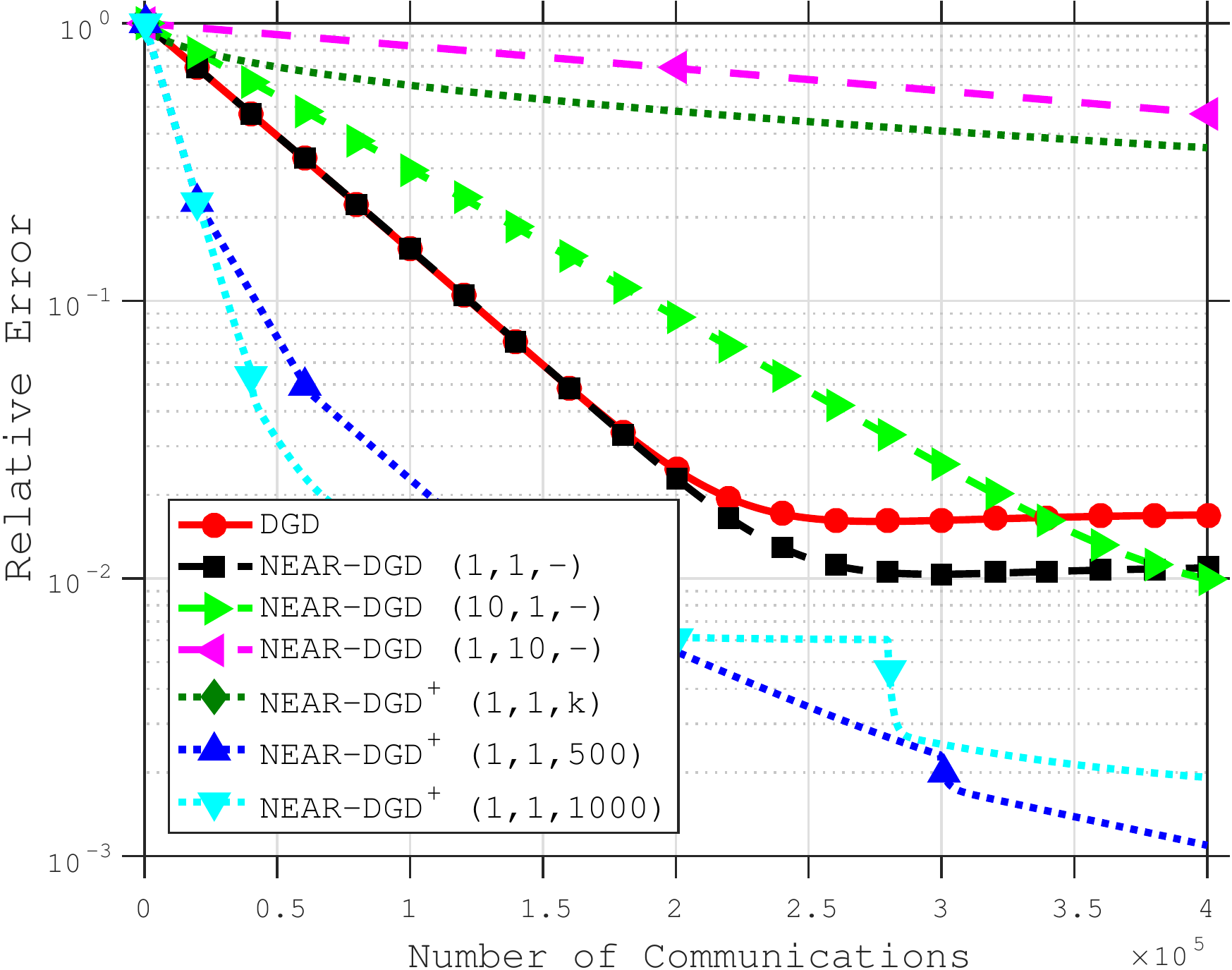}
\caption{ Performance of DGD, NEAR-DGD $(1,1,-)$, NEAR-DGD $(10,1,-)$, NEAR-DGD $(1,10,-)$, \ndgdp $(1,1,k)$, \ndgdp $(1,1,500)$, \ndgdp $(1,1,1000)$ measured in terms of relative error ($\|\bar{x}_k - x^\star\|^2/\| x^\star \|^2$) with respect to: (i) number of iterations, (ii) cost, (iii) number of gradient evaluations, and (iv) number of communications, on a quadratic problem ($n = 10$, $p=10$, $\kappa = 10^4$).}
\label{fig:ndgd_quad}
\end{figure}

Figure \ref{fig:ndgd_quad} shows the convergence rates of the methods. As predicted by the theory, it is evident that the methods that do not increase the number of consensus steps converge only to a neighborhood of the solution, whereas methods that increase the number of consensus steps converge to the solution. We note in passing that the \ndgd method (black line) converges to a better neighborhood that the DGD method (red line); this is not predicted by the theory, and is probably an artifact of the specific problem.

The \ndgdp method converges to the solution as predicted by the theory. In terms of number of iterations, the fastest method is the \ndgdp$(1,1,k)$ method. This is not surprising as the amount of work per iteration in this method is higher than that of the of \ndgdp$(1,1,500)$ and \ndgdp$(1,1,1000)$ (the practical variants). When comparing the methods in terms of the cost or the number of communications, the practical variants of the \ndgdp method perform significantly better. The cost metric in this case is a better indication of the performance of the method. In Figure \ref{fig:ndgd_quad} we assumed that the cost of a gradient evaluation and the cost of communication was the same, namely $c_c = c_g = 1$.


In Figure \ref{fig:ndgd_quad_diff_cost} we show the performance of the methods for different cost structures: (i) $c_c = 1, c_g = 10$; (ii) $c_c = 1, c_g = 1$; (iii) (i) $c_c = 10, c_g = 1$. The cost structures where $c_c \neq c_g$ arise in applications (problems) such as those described in Section \ref{sec:cost}. For example, the cost structure $c_c = 1, c_g = 10$ can be found in applications such as large scale machine learning problems on a cluster of physically connected computers, and the cost structure $c_c = 10, c_g = 1$ can be in found in applications such as controlling a swarm of battery powered robots.

\begin{figure}[]
\centering

\includegraphics[width=0.32\columnwidth]{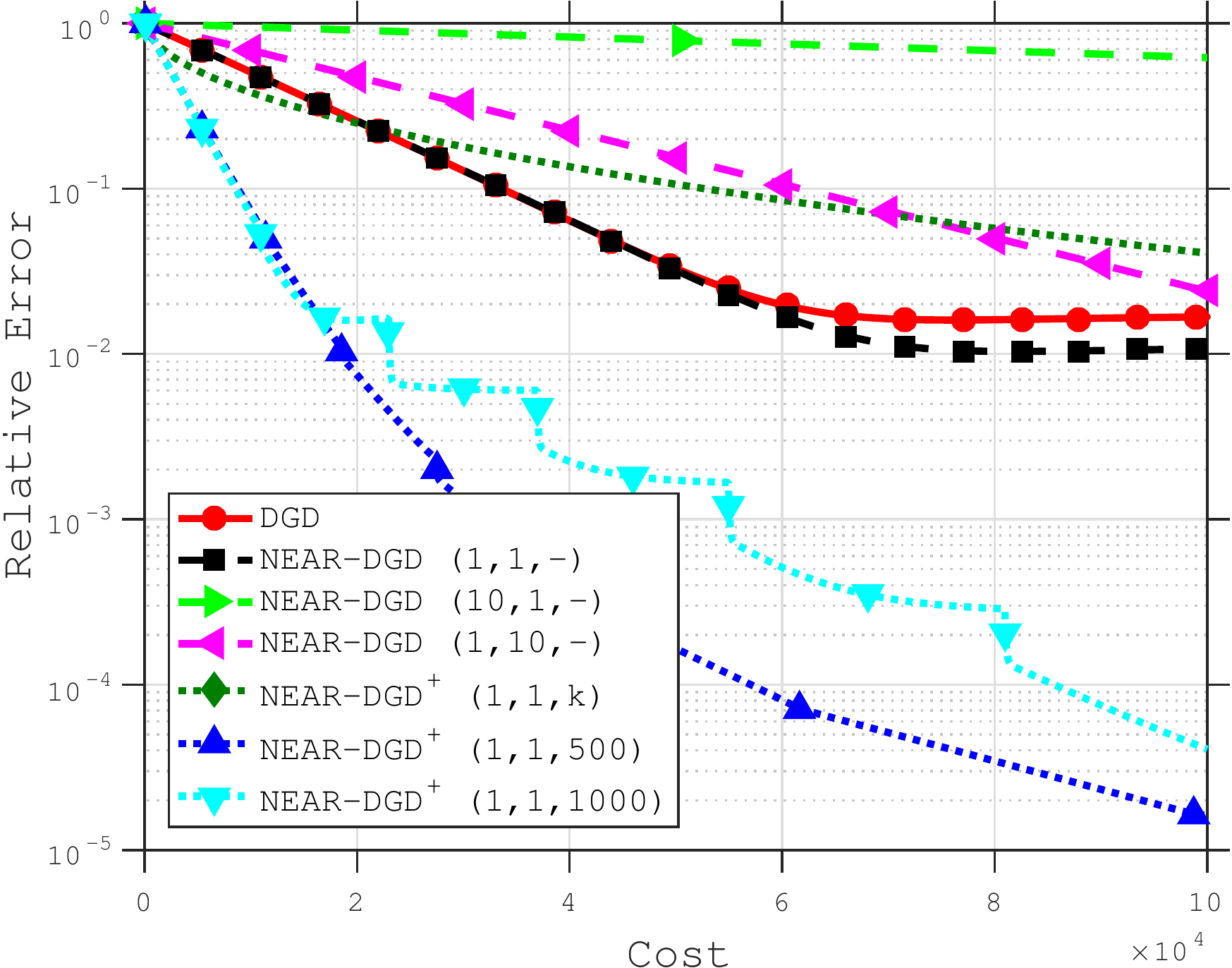}
\includegraphics[width=0.32\columnwidth]{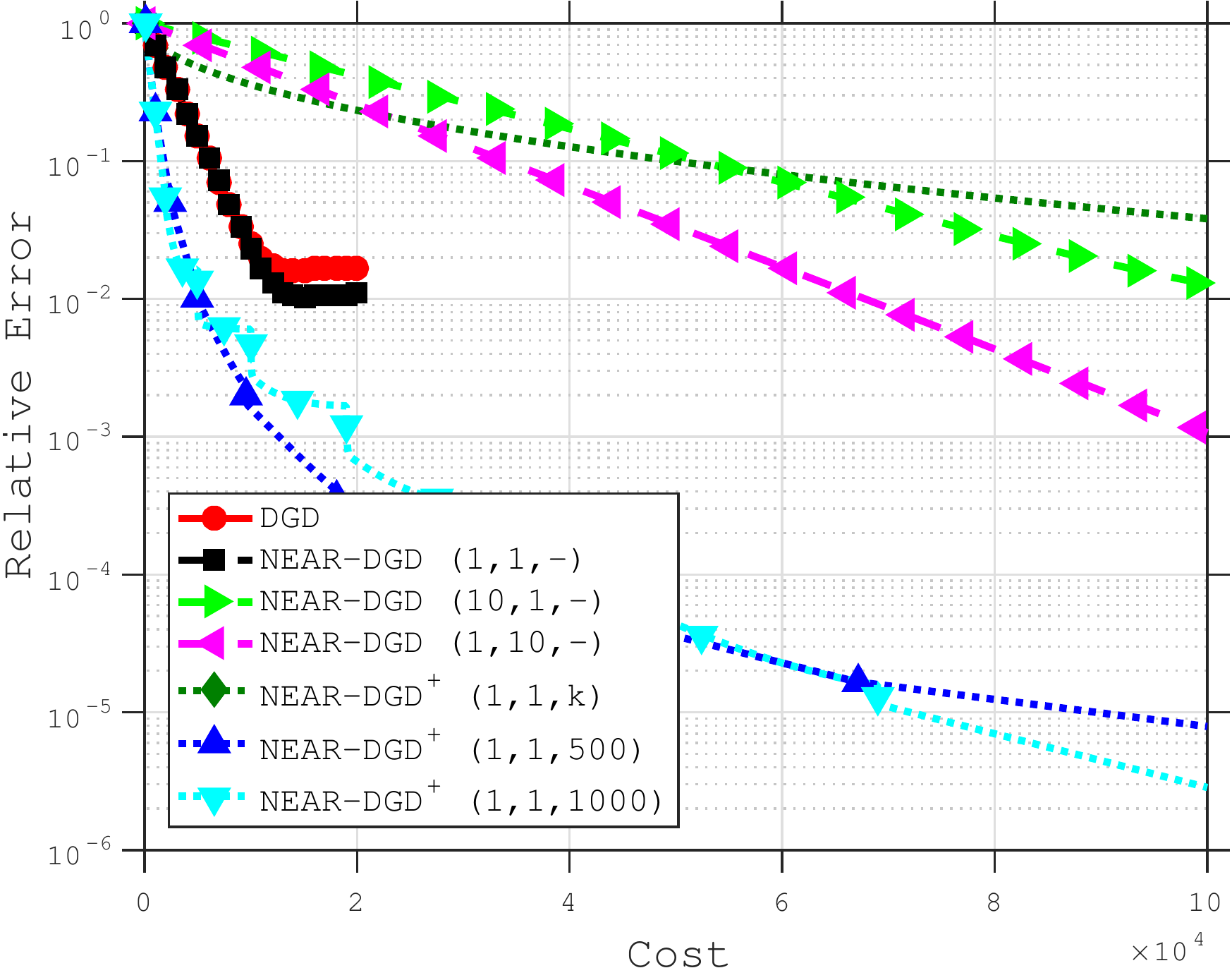}
\includegraphics[width=0.32\columnwidth]{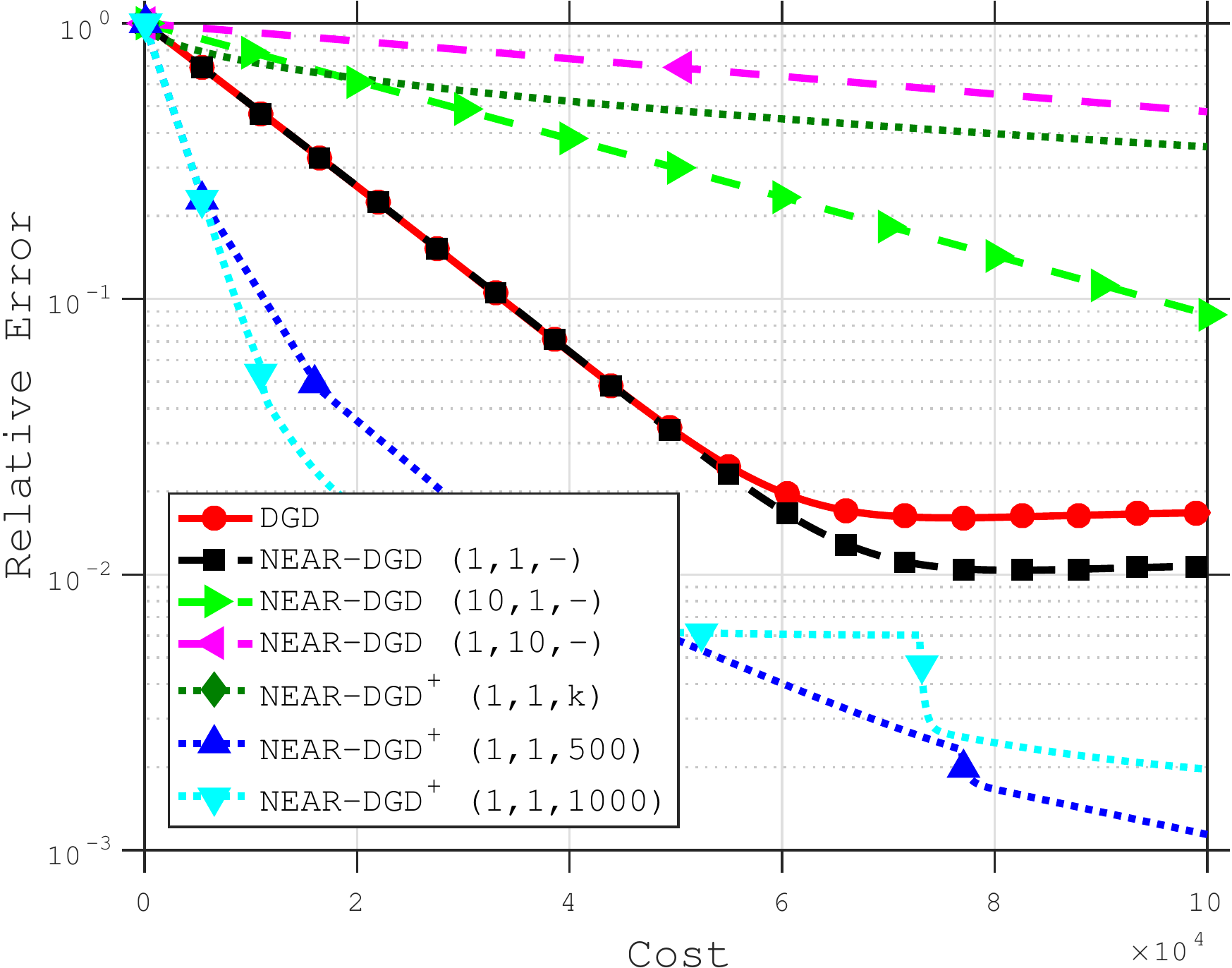}
\caption{Performance of DGD, NEAR-DGD (1,1,-), NEAR-DGD (10,1,-), NEAR-DGD (1,10,-), \ndgdp (1,1,k), \ndgdp (1,1,500), \ndgdp (1,1,1000) measured in terms of relative error ($\|\bar{x}_k - x^\star\|^2/\| x^\star \|^2$) with respect to different cost structures, on a quadratic problem ($n = 10$, $p=10$, $\kappa = 10^4$). {\textbf{Left}: $\textcolor{black}{c_{c}}=1$, $\textcolor{black}{c_{g}}=10$;} {\textbf{Center}: $\textcolor{black}{c_{c}}=1$, $\textcolor{black}{c_{g}}=1$;} {\textbf{Right}: $\textcolor{black}{c_{c}}=10$, $\textcolor{black}{c_{g}}=1$.}}
 \label{fig:ndgd_quad_diff_cost}
\end{figure}

Figure \ref{fig:ndgd_quad_diff_cost} shows that the performance of the methods is highly dependent on the specific cost structure of the application. Although in all cases the practical variants of the \ndgdp method perform the best in terms of the cost, the benefit of doing multiple consensus steps varies. On the left-most figure, the benefits are very apparent, whereas on the right-most figure, the benefits are not as apparent. That being said, of course, it is still the case that the methods that do not increase the number of consensus steps cannot converge to the solution.

\subsubsection{Binary Classification Logistic Regression Problems}

We now show numerical results illustrating the performance of the \ndgdt and \ndgdp methods on binary classification logistic regression problems that arise in machine learning. The objective function can be expressed as
\begin{align*}
		 f(x) = \frac{1}{n\cdot n_i} \sum_{i=1}^n \log(1 + e^{-(b_i)^T (A_ix)}) + \frac{1}{n \cdot n_i} \| x\|_2^2
\end{align*}
where $A \in \mathbb{R}^{n\cdot n_i \times p}$ and $b \in \{-1,1 \}^{n\cdot n_i }$, ($n$ denotes the number of nodes, $n_i$ denotes the size of the local data and $p$ is the dimension of the problem) and each node $i=1,...,n$ has a portion of $A$ and $b$, $A_i \in \mathbb{R}^{n_i \times p}$ and $b_i \in \mathbb{R}^{n_i}$. We report results for the \texttt{mushrooms} dataset ($n = 10$, $p=114$, $n_i = 812$) \cite{chang2011libsvm}, where the underlying graph is $4$-cyclic. Similar results were obtained for other standard machine learning datasets. For this experiment we set $c_c = 1, c_g = 1$.

\begin{figure}[]
\centering
\includegraphics[width=0.41\columnwidth]{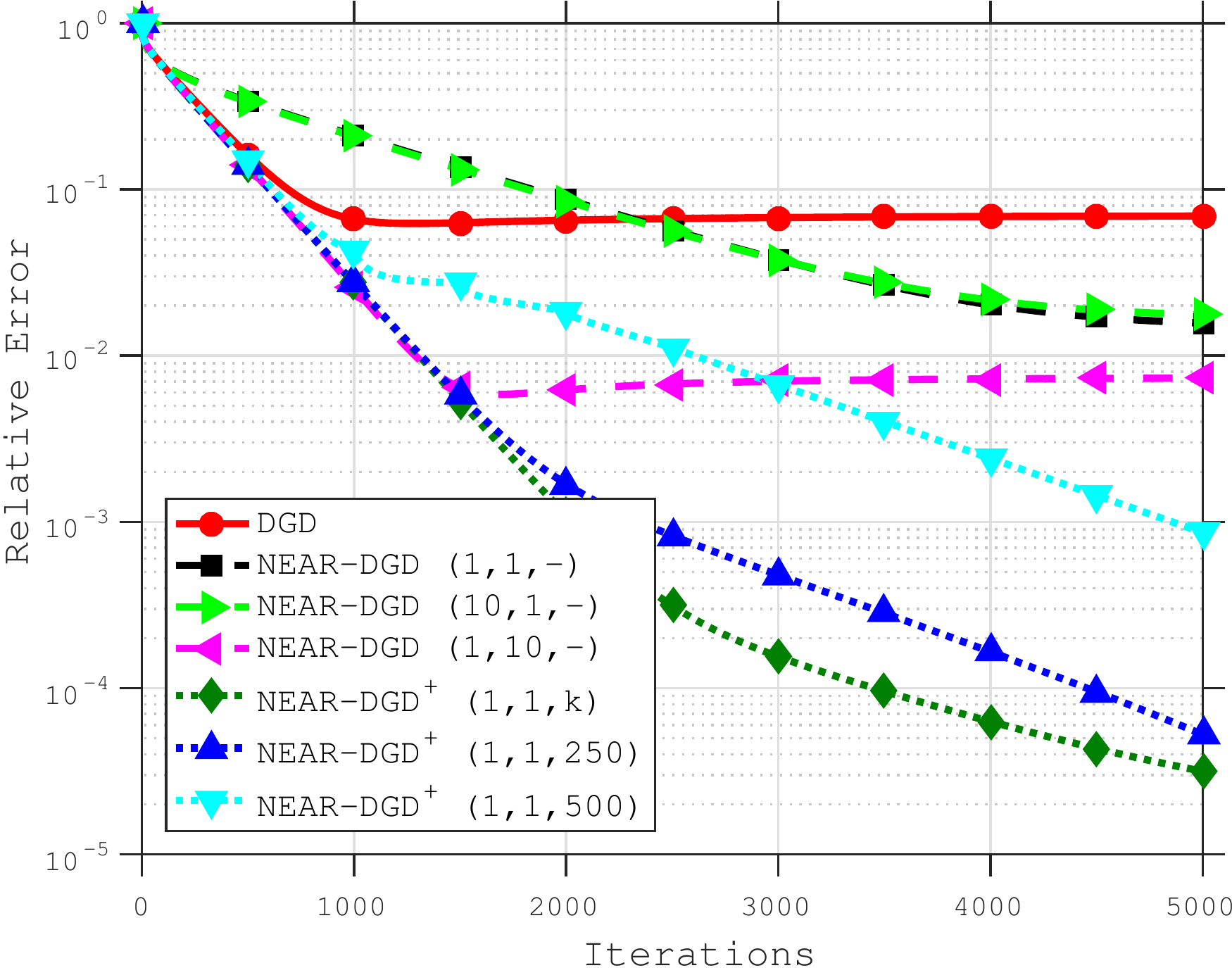}
\hspace{0.2cm}
\includegraphics[width=0.41\columnwidth]{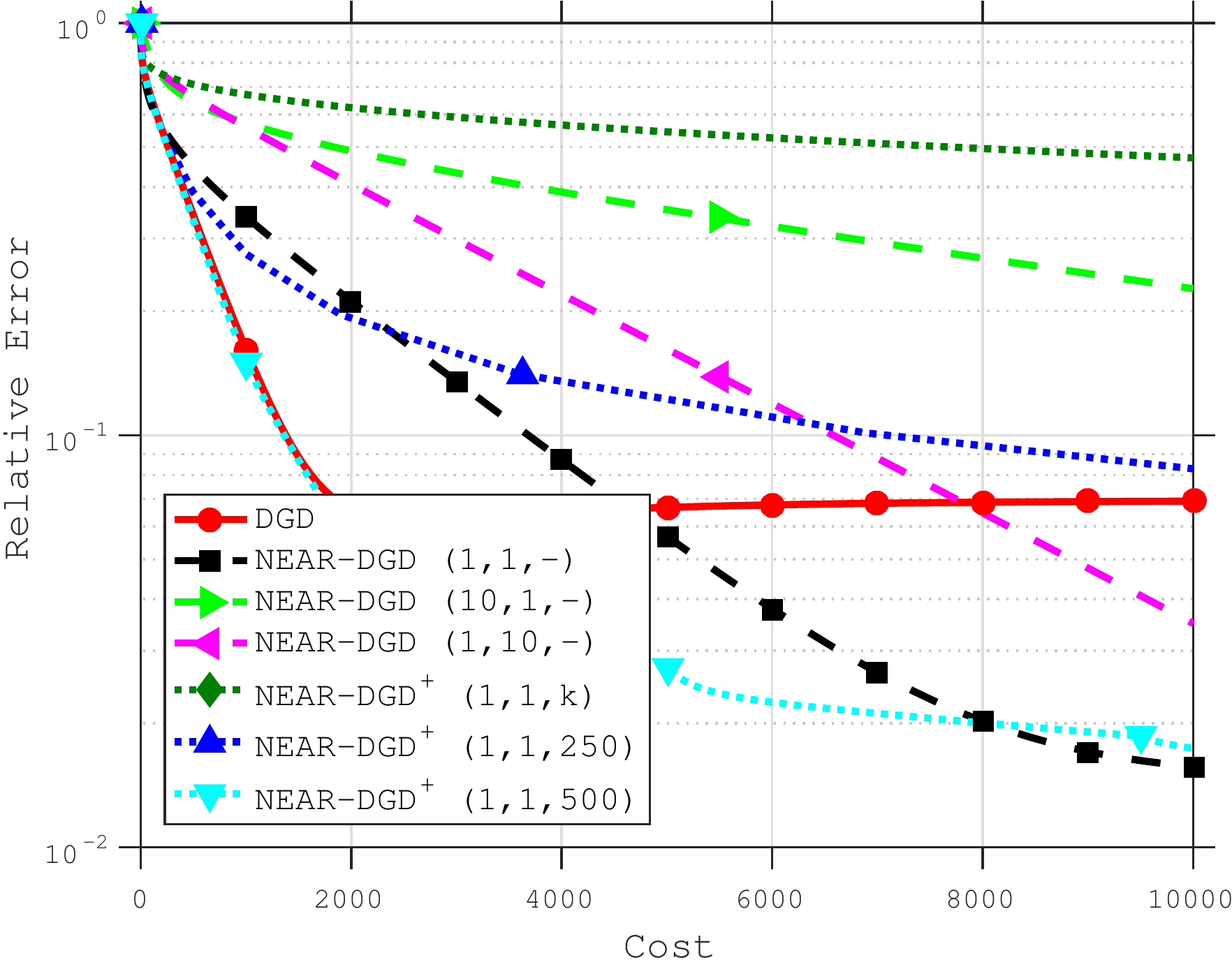}
\hspace{0.1cm}
\includegraphics[width=0.41\columnwidth]{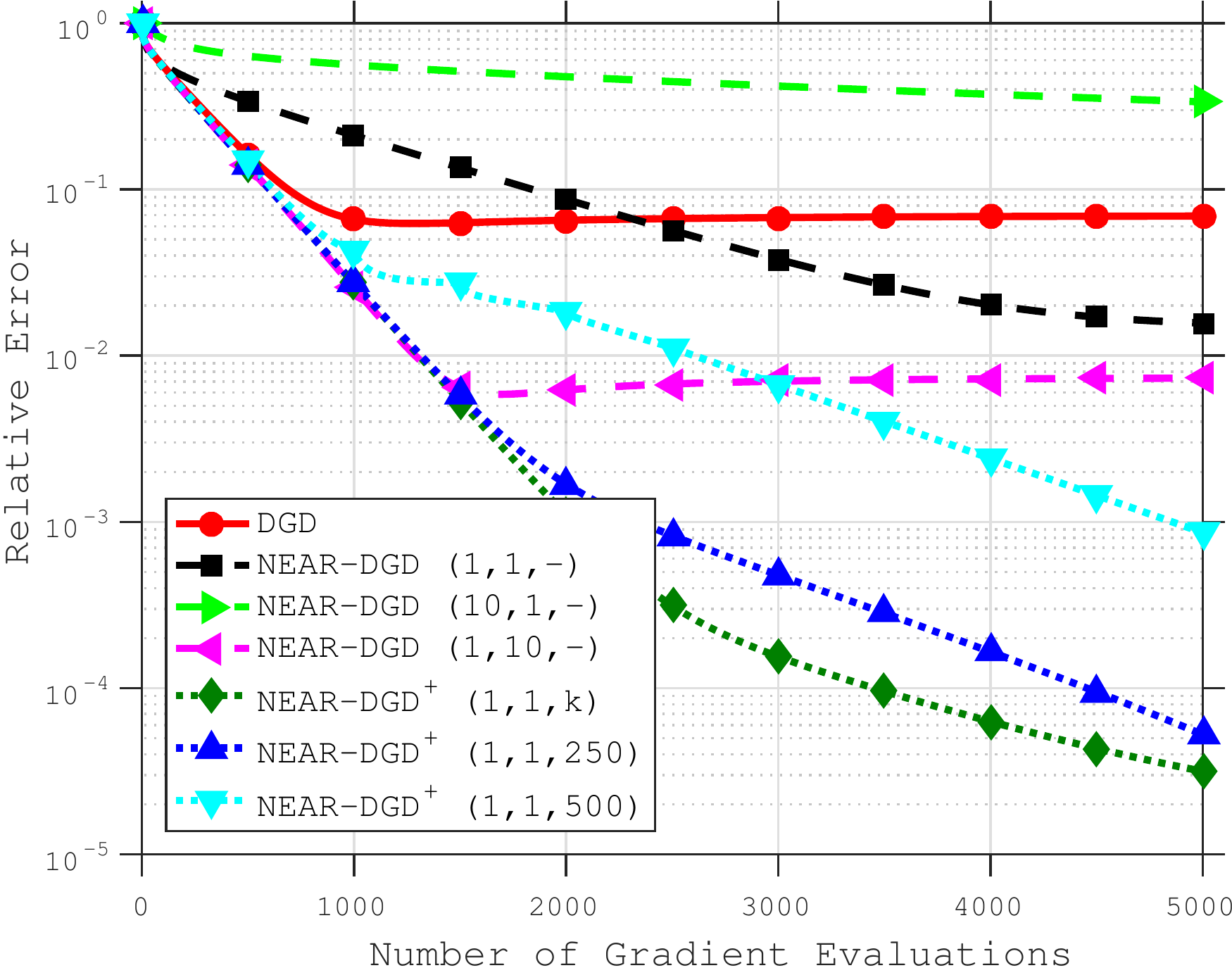}
\hspace{0.2cm}
\includegraphics[width=0.41\columnwidth]{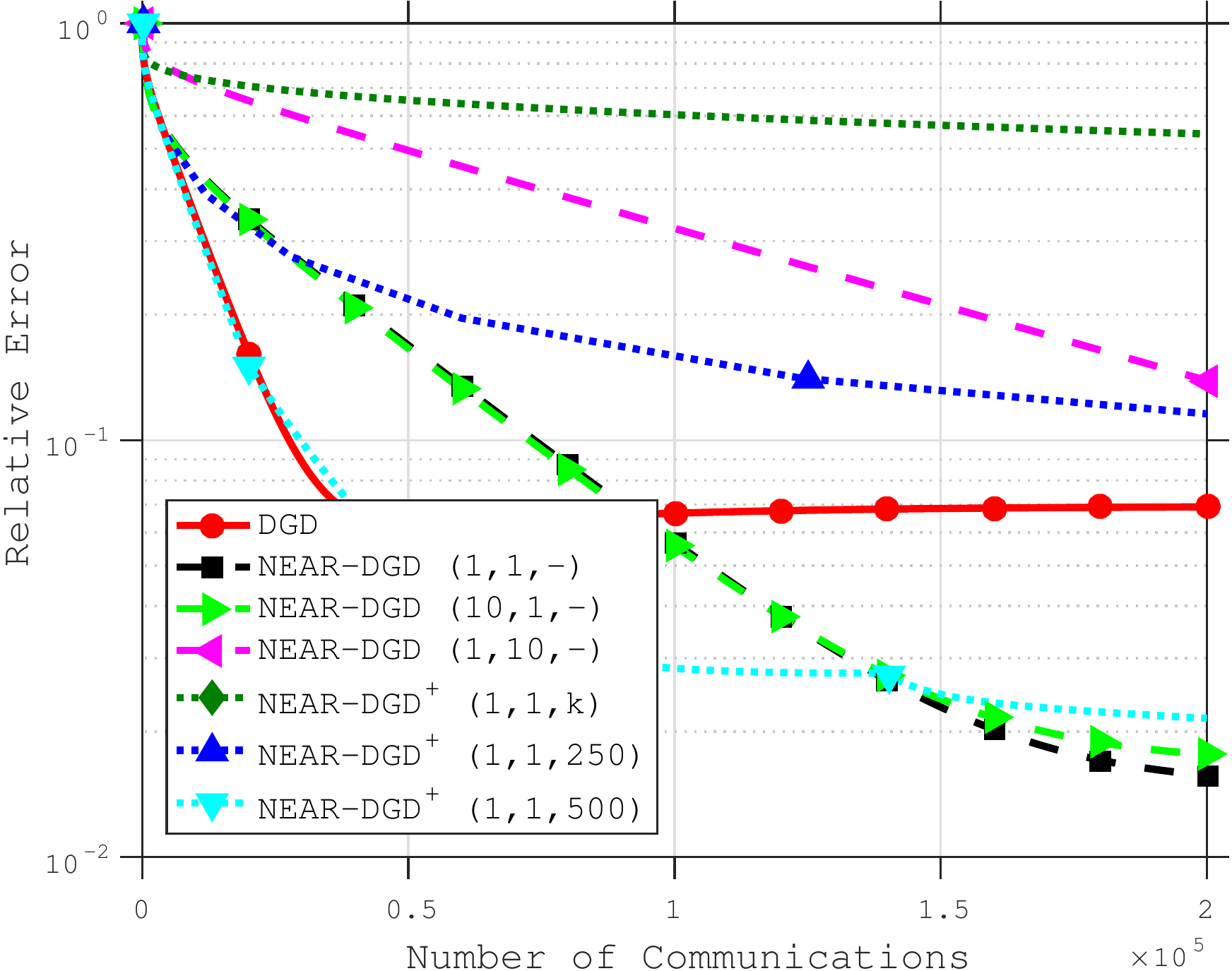}
\caption{Performance of DGD, NEAR-DGD $(1,1,-)$, NEAR-DGD $(10,1,-)$, NEAR-DGD $(1,10,-)$, \ndgdp $(1,1,k)$, \ndgdp $(1,1,500)$, \ndgdp $(1,1,1000)$ measured in terms of relative error ($\|\bar{x}_k - x^\star\|^2/\| x^\star \|^2$) with respect to: (i) number of iterations, (ii) cost, (iii) number of gradient evaluations, and (iv) number of communications, on a binary classification logistic regression problem (\texttt{mushroom}. $n = 10$, $p=114$, $n_i = 812$).}
\label{fig:log_reg}
\end{figure}

Figure \ref{fig:log_reg} illustrates the convergence rates of the $7$ methods. The results are similar to those for the quadratic problem. Namely, the variants that increase the number of consensus steps converge to the optimal solution whereas the other methods only converge to a neighborhood of the solution. Interestingly, the \ndgd(1,1,-) method is able to converge to a significantly better neighborhood that the base DGD method. Moreover, for a fixed budget (cost), it appears that the \ndgd(1,1,-) method is competitive with the \ndgdp variants. We should note that the figure plotting the error with respect to the cost does not show the outcome of the full experiment but rather only till the point that the DGD method terminated. However, looking at the per-iterations plots, the performance of the \ndgd(1,1,-) method stagnates whereas the performance of the \ndgdp methods does not.

\section{Final Remarks and Future Work}
\label{sec:fin_remarks}

In this paper, we propose an adaptive cost framework to evaluate the performance of distributed optimization methods. Given a specific application, our framework incorporates the costs associated with both communication and computation in order to evaluate the efficiency of distributed optimization methods. This work is a first step towards applying the proposed general cost framework. In particular, we study the well-known distributed gradient descent (DGD) method and decompose its communication and computation steps to construct three classes of more flexible methods: DGD$^t$, \ndgdt and NEAR-DGD$^+$.

The flexibility for each of these methods is illustrated by the fact that multiple consensus steps can be performed per gradient evaluation in environments where communication is relatively inexpensive. We show both theoretically and empirically that multiple consensus steps lead to better solution quality. We also design NEAR-DGD$^+$, an exact first order method, which increases the number of consensus steps as the algorithm progresses. As such, \ndgdp with a constant steplength converges to the optimal solution, as opposed to the standard DGD method that only converges to a neighborhood of the optimal solution.  Additionally, we show that for strongly convex functions, the \ndgdp method converges at a linear rate. Finally, through numerical experiments of different instances of these methods on quadratic and (binary classification) logistic regression problems, we illustrate the empirical performance of the methods and demonstrate the flexibility offered by our framework.

We should note that the same communication-computation decomposition approach can be applied seamlessly to many other distributed optimization methods, and this is a direction of future research that we wish to explore. Moreover, we plan to include other cost aspects into this framework, such as memory access, quantization and dynamic environments. Lastly, the question of how to automatically adjust the number of communication and computation steps, in an algorithmic way, depending on the environment, is a direction that we are currently investigating.

\ifCLASSOPTIONcaptionsoff
  \newpage
\fi



%
\bibliographystyle{IEEEtran}
\bibliography{balCC,citationSplitting}

%

\begin{IEEEbiography}[{\includegraphics[width=1in,height=1.25in,clip,keepaspectratio]{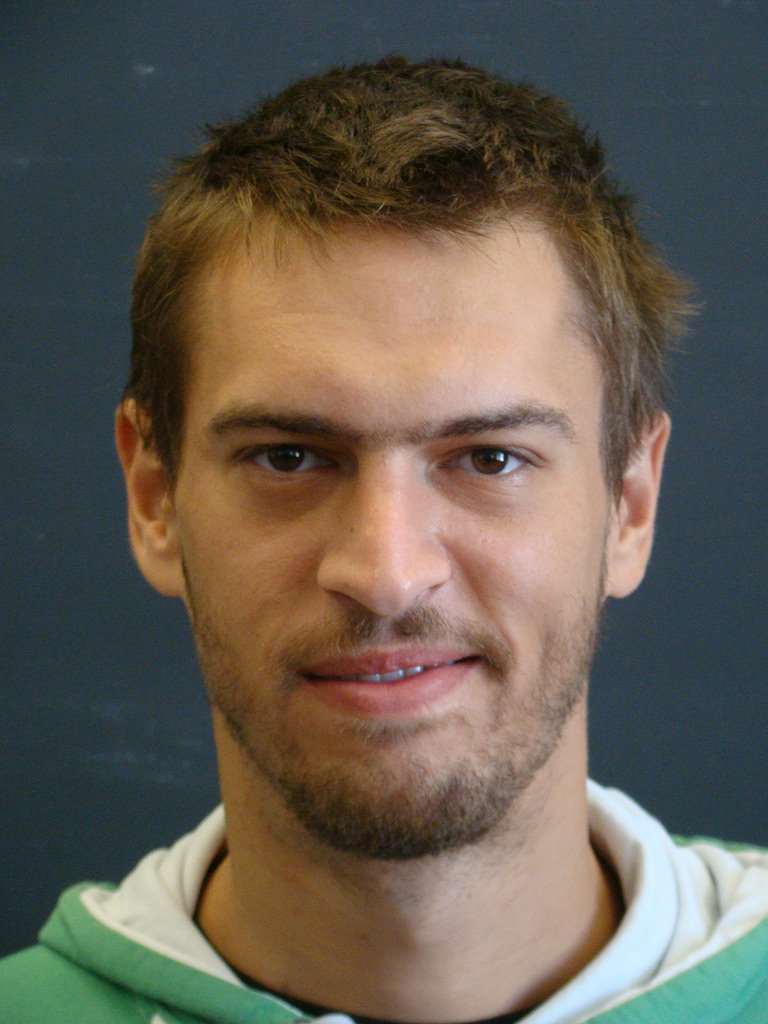}}]{Albert S. Berahas}
 is currently a Postdoctoral Research Fellow in the Industrial Engineering and Management Sciences (IEMS) Dept. at Northwestern University working under the supervision of Professor Jorge Nocedal. He completed his PhD studies in the Engineering Sciences and Applied Mathematics (ESAM) Dept. at Northwestern University in 2018, advised by Professor Jorge Nocedal. He received his undergraduate degree in Operations Research and Industrial Engineering (ORIE) from Cornell University in 2009, and in 2012 obtained an M.S. degree in Applied Mathematics from Northwestern University. Berahas has received the ESAM Outstanding Teaching Assistant Award, the Walter P. Murphy Fellowship and the John N. Nicholson Fellowship. Berahas' research interests include optimization algorithms for machine learning, convex optimization and analysis, derivative-free optimization and distributed optimization.
\end{IEEEbiography}

\begin{IEEEbiography}[{\includegraphics[width=1in,height=1.25in,clip,keepaspectratio]{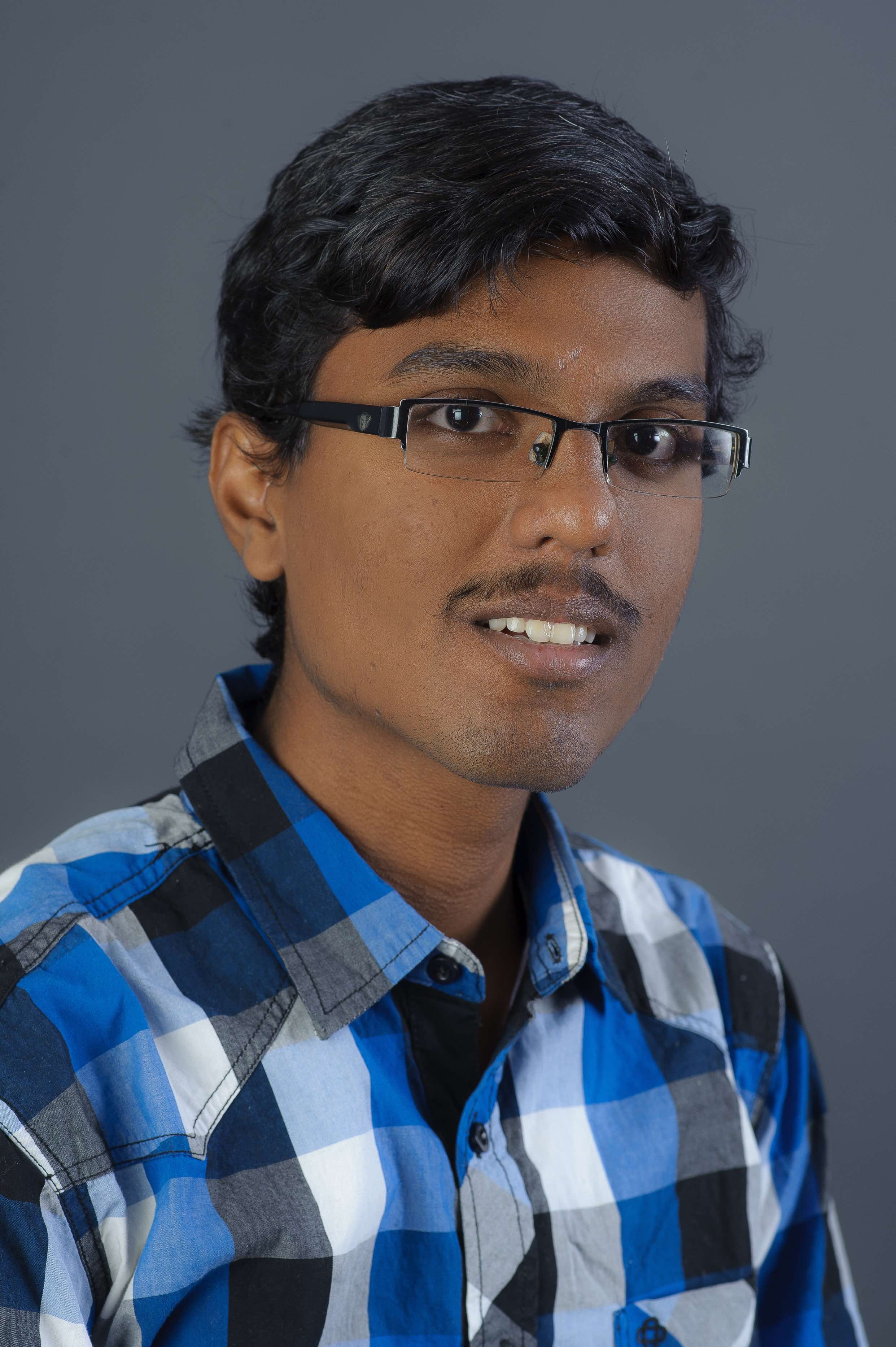}}]{Raghu Bollapragada}
is currently a PhD student in the Industrial Engineering and Management Sciences (IEMS) Dept. at Northwestern University, working under the supervision of Professor Jorge Nocedal. He received his undergraduate degree in Mechanical Engineering from Indian Institute of Technology (IIT) Madras, India in 2014. Bollapragada has received the IEMS Arthur P.  Hurter Award for outstanding academic excellence among first year graduate students and the Walter P. Murphy Fellowship at Northwestern University. Bollapragada's research interests include algorithms for machine learning, large-scale nonlinear optimization methods, convex optimization and analysis, stochastic optimization and distributed optimization.
\end{IEEEbiography}


\begin{IEEEbiography}[{\includegraphics[width=1in,height=1.25in,clip,keepaspectratio]{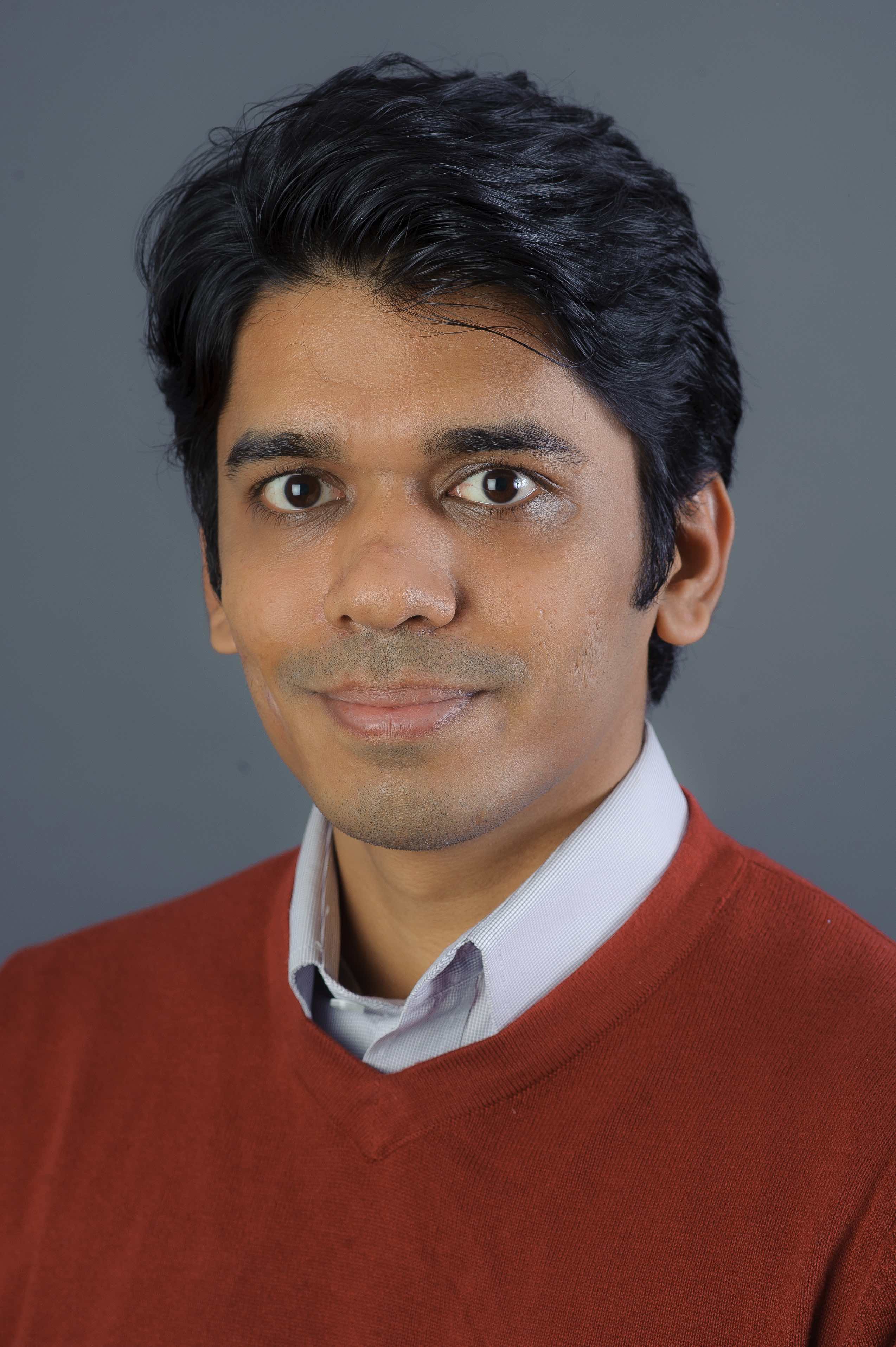}}]{Nitish Shirish Keskar}
is currently a Senior Research Scientist at Salesforce Research. He completed his PhD studies in Industrial Engineering and Management Sciences at Northwestern University in 2017. He was co-advised by Jorge Nocedal and Andreas Waechter. Prior to his PhD, Keskar received his undergraduate degree in Mechanical Engineering from VJTI in Mumbai, India. Keskar is the recipient of the 2017 best paper award for the Optimization Methods and Software journal for his paper on second-order method for L1 regularized convex optimization. Keskar's research interests include nonlinear optimization, deep learning and large-scale computing.
\end{IEEEbiography}

\begin{IEEEbiography}[{\includegraphics[width=1in,height=1.25in,clip,keepaspectratio]{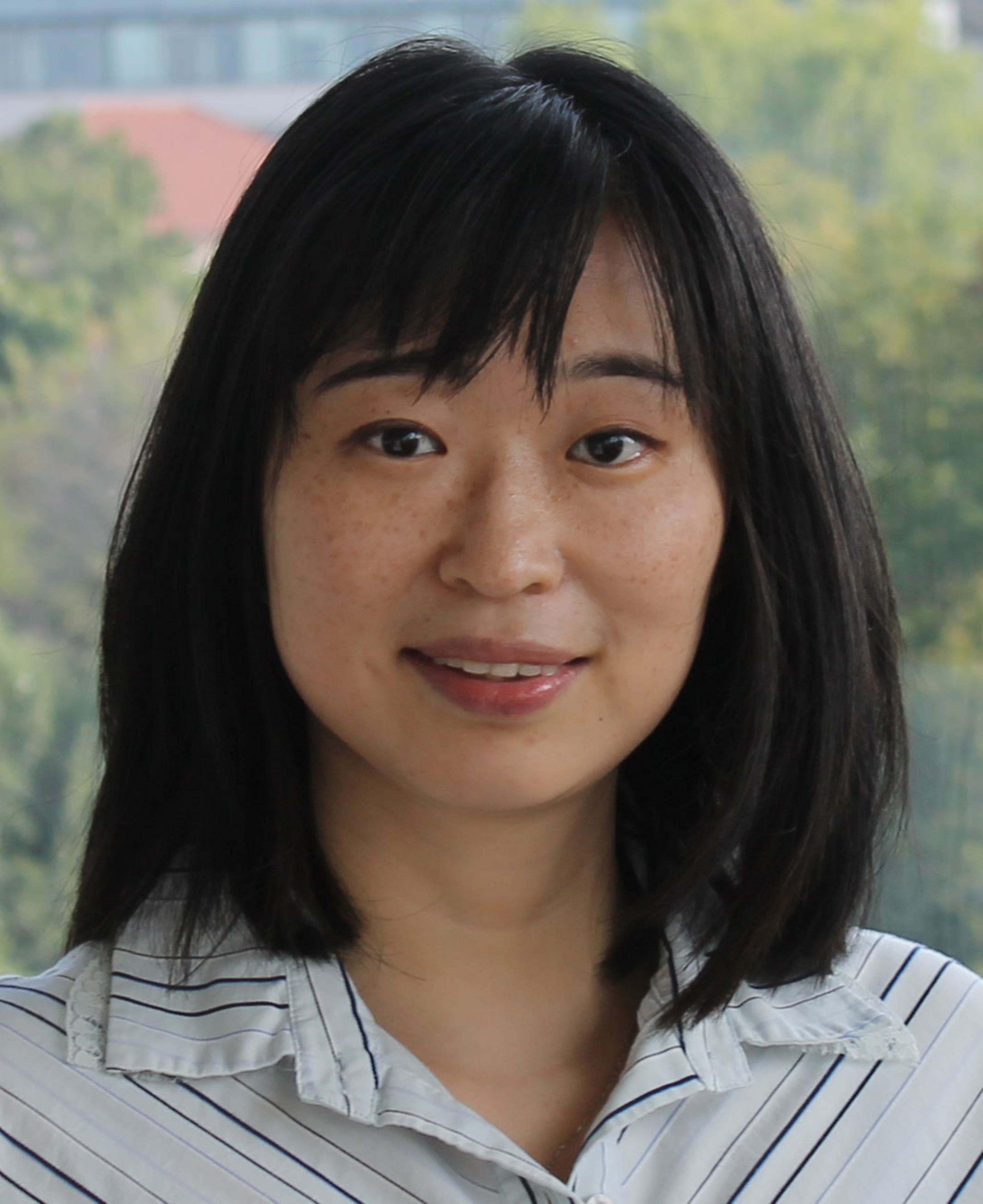}}]{Ermin Wei}
is currently an Assistant Professor at the EECS Dept. of Northwestern University. She completed her PhD studies in Electrical Engineering and Computer Science at MIT in 2014, advised by Professor Asu Ozdaglar, where she also obtained her M.S.. She received her undergraduate triple degree in Computer Engineering, Finance and Mathematics with a minor in German, from University of Maryland, College Park. Wei has received many awards, including the Graduate Women of Excellence Award, second place prize in Ernst A. Guillemen Thesis Award and Alpha Lambda Delta National Academic Honor Society Betty Jo Budson Fellowship. Wei'€™s research interests include distributed optimization methods, convex optimization and analysis, smart grid, communication systems and energy networks and market economic analysis.
\end{IEEEbiography}




\end{document}